\documentclass{amsart}

\usepackage{color}
\usepackage{txfonts} 
\usepackage{hyperref}
\usepackage{amscd, amsthm, amssymb, amsfonts, amsmath, enumerate}
\usepackage{mathrsfs} 

\usepackage[mathcal]{euscript}

\numberwithin{equation}{section}

\usepackage[all]{xy}
\SelectTips{cm}{} 

\newtheorem{theorem}{Theorem}[section]
\newtheorem{proposition}[theorem]{Proposition}

\newtheorem{lemma}[theorem]{Lemma}
\newtheorem{corollary}[theorem]{Corollary}

\theoremstyle{definition}
\newtheorem{definition}[theorem]{Definition}
\newtheorem{defn}[theorem]{Definition}

\newtheorem{example}[theorem]{Example}

\theoremstyle{remark}
\newtheorem{remark}[theorem]{Remark}

\numberwithin{equation}{section}

\newcommand{\N}{\mathbb{N}}

\newcommand{\R}{\mathbb{R}}

\newcommand{\M}{\mathcal{M}}

\newcommand{\Sp}{\mathcal{S}\mathit{p}}
\renewcommand{\Top}{\mathcal{T}\mathit{op}}
\newcommand{\Coll}{\mathcal{C}\mathit{oll}}
\newcommand{\Oper}{\mathcal{O}\mathit{per}}
\newcommand{\Alg}{\mathcal{A}\mathit{lg}}

\newcommand{\sG}{\mathcal{G}}
\newcommand{\sF}{\mathcal{F}}
\newcommand{\sP}{\mathcal{P}}
\newcommand{\sT}{\mathcal{T}}

\newcommand{\sH}{\mathcal{H}}

\newcommand{\sN}{\mathcal{N}}

\newcommand{\Com}{\mathscr{C}\!\mathit{om}}
\newcommand{\ulu}{\underline{u}}
\newcommand{\ulv}{\underline{v}}

\renewcommand{\emptyset}{\varnothing}

\renewcommand{\tilde}[1]{\widetilde{#1}}

\DeclareMathOperator{\colim}{colim}

\DeclareMathOperator{\Sym}{Sym}
\DeclareMathOperator{\Sing}{Sing}
\DeclareMathOperator{\Aut}{Aut}
\DeclareMathOperator{\res}{res}
\DeclareMathOperator{\Ind}{Ind}

\renewcommand{\hom}{\operatorname{Hom}}

\DeclareMathOperator{\ho}{Ho}

\newcommand{\po}{\ar@{}[dr]|(.7){\Searrow}}

\DeclareMathOperator{\Map}{Map}

\DeclareMathOperator{\ev}{Ev}

\newcommand{\cat}[1]{\mathcal{#1}}

\newcommand{\boxprod}{\mathbin\square}

\begin{document}

\title{Encoding equivariant commutativity via operads}

\author[J.J. Guti\'{e}rrez]{Javier J. Guti\'{e}rrez}
\address{Departament de Matem\`atiques i Inform\`atica, Facultat de
Matem\`atiques i Inform\`atica, Universitat de Bar\-ce\-lo\-na, Gran Via de les Corts Catalanes 585, 08007 Barcelona, Spain}
\email{javier.gutierrez@ub.edu}
\urladdr{http://www.ub.edu/topologia/gutierrez}

\author[D. White]{David White}
\address{Department of Mathematics, Denison University, Granville, OH 43023, USA}
\email{david.white@denison.edu} \urladdr{http://personal.denison.edu/~whiteda}


\begin{abstract}
In this paper, we prove a conjecture of Blumberg and Hill regarding the existence of $N_\infty$-operads associated to given sequences $\sF = (\sF_n)_{n \in \N}$ of families of subgroups of $G\times \Sigma_n$. For every such sequence, we construct a model structure on the category of $G$-operads, and we use these model structures to define $E_\infty^\sF$-operads, generalizing the notion of an $N_\infty$-operad, and to prove the Blumberg-Hill conjecture. We then explore questions of admissibility, rectification, and preservation under left Bousfield localization for these $E_\infty^\sF$-operads, obtaining some new results as well for $N_\infty$-operads.

\end{abstract}

\maketitle

\section{Introduction}

The work of Hill, Hopkins, and Ravenel on the Kervaire invariant one problem in \cite{kervaire} conclusively demonstrates the value of equivariant spectra to modern stable homotopy theory, and in particular of equivariant commutative ring spectra. The computations used in \cite{kervaire} rely on the slice spectral sequence and the existence of multiplicative norm functors on the category of equivariant commutative ring spectra. For a compact Lie group $G$, a genuinely commutative ring $G$-spectrum has multiplicative norm functors parameterized by closed subgroups $H < G$, by natural numbers $n$, and by homomorphisms $\rho: H \to \Sigma_n$ to the symmetric group on $n$ letters.

Blumberg and Hill \cite{blumberg-hill} introduced $N_\infty$-operads to encode equivariant algebraic structure, including multiplicative norm maps. These operads interpolate between the $E_\infty$-operad in spaces (which encodes no multiplicative norms) and $E_\infty$-$G$-operads (which encode all possible norms). An $N_\infty$-operad is a $G$-operad $P$ such that $P(0)$ is $G$-contractible, the action of $\Sigma_n$ on $P(n)$ is free, and $P(n)$ is the universal space for a family $\sN_n(P)$ of subgroups of $G\times \Sigma_n$ containing all subgroups of the form $H\times 1$. The condition that $\Sigma_n$ acts freely on $P(n)$ implies that, for every $\Gamma \in \sN_n(P)$, $\Gamma \cap (1\times \Sigma_n) \cong 1 \times 1$.

In this paper, we generalize the notion of an $N_\infty$-operad to what we call $E_\infty^\sF$-operads, where $\sF = (\sF_n)$ is a sequence of families of subgroups of $G\times \Sigma_n$, not necessarily satisfying the requirement of the families $\sN_n(P)$ above. For each sequence $\sF$, we construct a model structures on the category of $G$-operads, and we use these model structures to construct our $E_\infty^\sF$-operads. Using these model structures, we then prove a conjecture of Blumberg and Hill regarding the existence of $N_\infty$-operads. We also work out questions of admissibility, rectification, and strictification for these $E_\infty^\sF$-operads, generalizing results from \cite{blumberg-hill}. Throughout the paper, $G$ is taken to be a compact Lie group, except in Sections \ref{sec:admissibility} and \ref{sec:rectification}, where we restrict to finite groups $G$ in order to use the techniques from \cite{kervaire} to transfer model structures to categories of algebras over $N_\infty$-operads.

Blumberg and Hill give several examples of $N_\infty$-operads, but do not prove that, for every sequence $\sN = (\sN_n)_{n\in \N}$ of families of subgroups, there is an associated $N_\infty$-operad. Indeed, this is not true in general, because the operad composition maps place certain restrictions on the families. Blumberg and Hill conjectured that these restrictions are the only obstacle to the existence of an $N_\infty$-operad associated to a given sequence $\sN$. The main result of this paper, is a proof of this conjecture (see Section \ref{sec:cof-rep-operads}), by identifying the precise relationships between the $\sN_n$ in order for an $N_\infty$-operad $P$, with $P(n)$ a universal space for the family $\sN_n$, to exist. In a 2017 preprint \cite{rubin}, Rubin also verifies this conjecture using different methods (related to indexing systems). Bonventre and Pereira \cite{bonventre-pereira} have an alternative approach, based on equivariant trees.

After a review of model categories, operads, and equivariant operads in Section~\ref{sec:prelim}, we develop $\sF$-fixed point model structures on the category of $G$-operads in
Section~\ref{sec:model-operads}, and then we realize each operad
$E_\infty^\sF$ as a cofibrant replacement for the commutative operad $\Com$
in the $\sF$-fixed point model structure on $G$-operads.
In Section \ref{sec:cof-rep-operads}, we introduce the notion of a \emph{realizable} sequence $\sF$ as a sequence satisfying a condition relating the families $\sF_n$ for different $n$. We then resolve the Blumberg-Hill conjecture, by proving that this condition is equivalent to the existence of an $N_\infty$-operad $P$ whose spaces $P(n)$ are universal spaces for the families $\sF_n$.

In Section \ref{sec:admissibility}, we prove that, for certain sequences $\sF$, there is a transferred model structure on algebras over an $E_\infty^\sF$-operad (in $G$-spaces and, for finite $G$, in $G$-spectra). In Section \ref{sec:rectification}, we address the question of rectification between algebras over different $E_\infty^\sF$-operads, and we prove that in an appropriate model structure on $G$-spectra there is a Quillen equivalence between certain $E_\infty^\sF$-algebras and strictly commutative ring spectra. In Section \ref{sec:localization} we discuss two examples that demonstrates that Bousfield localization can reduce $E_\infty^\sF$-structures to $E_\infty$-structures (i.e., with no multiplicative norm maps) or even less structure. We then characterize the localizations that preserve $E_\infty^\sF$-structures, and in doing so generalize the example to demonstrate localizations which reduce structure to $E_\infty^\sF$ for all $\sF$. In Appendix \ref{appendix}, we verify the model categorical conditions required so that left Bousfield localizations of $G$-spectra exist.

\bigskip

\textbf{Acknowledgments:} A draft of this paper has been circulating in various forms since 2013. In the summer of 2013, we benefited from an extensive correspondence with Andrew Blumberg and Mike Hill, wherein they kindly explained their program regarding $N_\infty$-operads, and made several helpful observations related to our operads $E_\infty^\sF$. In 2014, we also benefited from correspondence with Justin Noel, which improved Section \ref{sec:localization}.
We are also grateful to John Greenlees for catching a mistake in a 2014 version of this document, and giving us the hint that led us to the notion of a realizable sequence and Theorem \ref{thm:N-infty-operads-exist} in 2015. We thank Doug Ravenel for suggesting we work in the positive complete model structure on $G$-spectra, and for asking the questions that led us to write down the proof of Proposition \ref{prop:G-spectra-cellular}. We thank Jonathan Rubin and Peter Bonventre for (independently) asking in 2016 if Theorem \ref{thm:N-infty-operads-exist} could be made an ``if and only if.'' We would also like to thank Peter May for consistently encouraging the second author to finish this paper. Finally, we thank the IMUB for supporting visits of the second author to Barcelona in 2013 and 2017, Ieke Moerdijk and Radboud University for supporting the second author to visit Nijmegen in 2013 and 2015, and Carles Casacuberta for facilitating the start of our collaboration in 2013. This draft was improved by several helpful comments from an anonymous referee, from Jonathan Rubin, and from Benjamin B\"{o}hme.

\section{Preliminaries} \label{sec:prelim}

In this section we give a review of model categories, operads, and
equivariant operads.
\subsection{Model Categories}

We will assume the reader is familiar with the language of model structures.
Excellent treatments are given in \cite{hirschhorn} and \cite{hovey-book} .

All of our model categories $\M$ will be \emph{cofibrantly generated}, that is,
there exist a set $I$ of generating cofibrations and a set $J$ of generating
trivial cofibrations that one can use to perform the small object argument
(see \cite[Definition 11.1.2]{hirschhorn} or \cite[Definition~2.1.17]{hovey-book} for a precise definition). An object $X$ is \emph{small} relative to a class of maps $K$, if there is a large enough cardinal $\kappa$ such that the functor $\M(X,-)$ commutes with $\kappa$-filtered colimits of maps in $K$. 

If $\M$ is a model category and $X$ is an object in $\M$ then $QX$ will
denote the (functorial) cofibrant replacement of $X$, and $RX$ will denote
the (functorial) fibrant replacement. Thus, there is a canonical trivial
fibration $q_X:QX \to X$ and a canonical trivial cofibration $r_X:X\to RX$.

Recall that a \textit{monoidal model category} $\M$ is a model category with
a monoidal product $\otimes$ and monoidal unit $S$ such that the following
two conditions are satisfied:

\begin{itemize}
\item The \emph{pushout product axiom}: if $f: A\to B$ and $g: X\to Y$ are cofibrations,
    then the pushout product
$$
f \boxprod g\colon A\otimes Y \coprod_{A\otimes X} B\otimes X \longrightarrow
    B\otimes Y
$$
is also a cofibration. Furthermore, $f\boxprod g$ is a trivial
cofibration if either $f$ or $g$ is a trivial cofibration.
\item The \emph{unit axiom}: if $X$ is cofibrant, then the map
$$
q_S\otimes id_X\colon QS\otimes X \longrightarrow S\otimes X\cong X
$$
is a weak equivalence.
\end{itemize}

An additional axiom on a monoidal model category which we will have cause to
consider is the \textit{monoid axiom}, which first appeared in
\cite[Definition 3.3]{SS00} and guarantees that the categories of monoids in
a monoidal model category inherit a model structure with weak equivalences
and fibrations created by the forgetful functor. Recall that given a class of maps $\mathcal{I}$, the \emph{saturated class} generated by $\mathcal{I}$, also denoted by $\mathcal{I}$-cell, is smallest class of maps that contains $\mathcal{I}$ and it is closed under pushouts and transfinite compositions. An object is an \emph{$\mathcal{I}$-cell complex} if the map from the initial object to it is in $\mathcal{I}$-cell.
\begin{itemize}
\item The \emph{monoid axiom}: for all objects $X$, the saturated class
    generated by the maps $X\otimes f$ where $f$ runs through all trivial
    cofibrations is contained in the weak equivalences.
\end{itemize}

\subsection{Spaces}

Let $\Top$ denote the category of compactly generated weak-Hausdorff spaces. This is a monoidal model category~\cite[Proposition 4.2.11]{hovey-book}. Let $G$ be a compact Lie group, and let $\Top^G$ denote the catetory of spaces with a $G$-action, and with equivariant maps. 
When working in the pointed setting, the $G$-action is assumed to fix the distinguished basepoint.

A map $f$ in $\Top^G$ is a weak equivalence (resp. fibration) if the $H$ fixed points $f^H$ is a weak equivalence (resp. fibration) in $\Top$ for all closed subgroups $H\le G$. We need $H$ to be closed so that the quotient topology on $G/H$ is weak Hausdorff. For this reason, whenever we consider subgroups in this paper, we will always assume they are closed. The fixed points functor $(-)^H$ has a left adjoint $G/H\times (-)\colon \Top \to \Top^G$, and the generating (trivial) cofibrations have the form $G/H \times i$ for all subgroups $H\le G$, where $i$ is a generating (trivial) cofibration in $\Top$. The pointed analogue works the same way.

Turning now to monoidal structures, recall that $\Top^G$ is closed symmetric
monoidal with the Cartesian product, where we use the diagonal action of $G$
on $X\times Y$, and the conjugation action of $G$ on the equivariant mapping
space $\Map_G(X,Y)$, that is, $(g\cdot f)(x)=g\cdot f(g^{-1}\cdot x)$. For pointed spaces, the smash product is used. The category $\Top^G$ is enriched, tensored, and cotensored over $\Top$. 
Recall that $\Top^G$ is a proper, monoidal model category and a \emph{topological model category} in the sense of \cite[Definition~4.2.18]{hovey-book}, meaning that for every cofibration $f: A\to B$ in $\Top$ and every cofibration $g: X\to Y$ in $\Top^G$, the colimit map $f \boxprod g\colon A\otimes Y \coprod_{A\otimes X} B\otimes X \longrightarrow B\otimes Y$ is also a cofibration (where $\otimes$ denotes the tensoring over $\Top$), and is a trivial cofibration if either $f$ or $g$ is a trivial cofibration (this is explained in \cite[Section 2.3]{fausk-equivariant-pro-spectra} and \cite[Theorem IV.6.5]{mandell-may-equivariant}, among other places). Any topological model category is automatically a simplicial model category, by applying the Sing functor to the topological mapping spaces.

There are also variants of this model structure relative to a family of subgroups of $G$. Families of subgroups are crucial to the study of equivariant homotopy theory; they are necessary for the definition of the geometric fixed points functor, they come up several times in constructions of free spectra~\cite{lewis-may-steinberger}, and they are related to the Baum--Connes and Farrell--Jones conjectures. We refer the reader to the excellent survey articles \cite{bohmann-overview} and \cite{luck-survey-classifying-families} for more information on the importance of families. Fixed-point model structures allow for the homotopical study of the information which can be ``seen'' by a
family, and we will exploit this point of view throughout the paper.

For any set $\sF$ of subgroups of $G$ which contains the trivial subgroup, there is a cofibrantly generated model structure on $\Top^G$ in which a map $f$ is a weak equivalence (resp. fibration) if and only if $f^H$ is a weak equivalence (resp. fibration) in $\Top$ for all $H \in \sF$; see for instance~\cite[Section 3]{stephan} for a general approach to equivariant model structures.
We will denote this model structure by $\Top^\sF$. The generating (trivial) cofibrations are
given by the set of maps of the form $G/H \times g$, where $g$ is a generating (trivial) cofibration of topological spaces, and $H\in \sF$. We will assume that $\sF$ contains the trivial subgroup
$\{e\}$ in order to ensure compatibility with the model structure on $\Top$.

If $K$ is conjugate to $H$, then $G/K$ is isomorphic as a $G$-space to $G/H$, so nothing is lost by assuming $\sF$ is closed under conjugation. Without a further hypothesis on $\sF$, however, we will not know this model structure is monoidal, even when $G$ is a finite group. In particular, the pushout product axiom requires $G/H \times G/K$ with the diagonal action to be cofibrant. We will additionally assume that the set $\sF$ is closed under subgroups, in particular, closed under intersections. 

\begin{definition}
A collection $\sF$ of subgroups of $G$ is called a \emph{family} if it is closed under subgroups and conjugation. 
\end{definition}
Note that some other authors require $\sF$ to be closed under subconjugacy, that is, if $H\in \sF$ and $g^{-1}Kg\subset H$ then $K\in \sF$. However, for our purposes---the existence of a monoidal model structure relative to $\sF$ on $G$-spaces and $G$-spectra, and the existence of universal classifying spaces \cite[Section 1.2]{luck-survey-classifying-families}---it is sufficient to assume $\sF$ is closed under subgroups.

In the case where $G$ is a compact Lie group rather than a finite group, additional care must be taken to ensure the model structure on $\Top^{\sF}$ satisfies the pushout product axiom. As observed in~\cite[Lemma 2.9]{fausk-equivariant-pro-spectra}, the key condition to assume on $\sF$ is that it forms an \textit{Illman collection}. This means that $(G/H \times G/K)_+$ is an $\sF I$-cell complex for any $H,K\in \sF$ \cite[Theorem 5.5]{illman}. Justin Noel has pointed out (private correspondence) that any family of closed subgroups of a compact Lie group satisfies Illman's condition, since $G/K$ is a compact $H$-manifold and hence admits an $H$-equivariant triangulation, making it into an $H$-CW complex. It follows that $\Ind_H^G \res_H^G G/K\cong G/H\times G/K$ is a $G$-CW complex.

We summarize these considerations as:

\begin{proposition} \label{prop:family-top}
If $G$ is a finite group and $\sF$ is a family of subgroups of $G$, or if $G$
is a compact Lie group and $\sF$ is an Illman collection, then $\Top^{\sF}$
is a proper, topological, closed symmetric monoidal model category. The weak equivalences
(resp. fibrations) are the maps $f\colon X\to Y$ such that $f^H\colon X^H\to Y^H$ is a weak
equivalence (resp. fibration) in $\Top$, for every $H$ in $\sF$, respectively. The generating (trivial) cofibrations are maps of the form $G/H \times i$ where $i$ is a generating (trivial) cofibration in $\Top$, and $H \in \sF$.
\end{proposition}

Of course, there is also a pointed analogue for $\Top^G_*$.

\begin{remark}
The category of compactly generated spaces is not a combinatorial model category, because not all spaces are small; see \cite[Section 2.4]{hovey-book}. One could work in a combinatorial model for $\Top^G$ by using Jeff Smith's $\Delta$-generated spaces. Alternatively, one could work with simplicial sets, but then one would need to either consider actions of $\Sing(G)$ or restrict to discrete groups $G$.
\end{remark}

Lastly, we review universal spaces of families of subgroups, as these will be required in Section \ref{sec:cof-rep-operads}. For any group $G$ and any family of subgroups $\sF$, the \emph{universal space of principal $G$\nobreakdash-bundles for the family $\sF$} is a space $E_{\sF}G$ characterized (up to $G$-equivariant weak equivalence) by the following properties (see \cite[Section 1.2]{luck-survey-classifying-families}):
\begin{itemize}
\item All isotropy groups of $E_{\sF}G$ belong to $\sF$, or equivalently, $(E_{\sF} G)^H=\emptyset$ if $H\not\in\sF$, where $(-)^H$ denotes $H$-fixed points.
\item $(E_{\sF}G)^H\simeq *$ for all $H\in\sF$.
\end{itemize}
The existence of CW-models for these spaces $E_\sF G$ is verified in \cite{luck-uribe} for all groups and families considered in this paper (i.e., for $G$ a compact Lie group). Existence can also be deduced from \cite[Remark 6.5]{piacenza}.

Now, let $\Sigma$ be another group and consider the product $G\times\Sigma$.
Given a family of subgroups $\sF$ of $G\times\Sigma$, let $E_{\sF}(G\times \Sigma)$ be the corresponding universal space for principal $(G\times \Sigma)$\nobreakdash-bundles. If $\sF$ is the family of all subgroups
$K\le G\times\Sigma$ such that $K\cap(\{e\}\times \Sigma)=\{e\}$, then
$E_{\sF}(G\times \Sigma) = E_G(\Sigma)$ is the universal space for $G$-equivariant principal $\Sigma$-bundles. Observe that a subgroup $K\le G\times\Sigma$ satisfies the condition $K\cap(\{e\}\times \Sigma)=\{e\}$ if and
only if $K$ is of the form
$$
\Gamma(\rho)=\{(h, \rho(h)) \mid h\in H,\ \rho\colon H\to \Sigma \},
$$
where $H$ is a subgroup of $G$.

As explained in~\cite[Section 2.6]{cort_ellis}, a CW-model for $E_{\sH}G$ is a cofibrant replacement of the one point space $*$ in the $\sH$-model structure of spaces, for any family $\sH$. Similarly $E_G(\Sigma)$ is a cofibrant replacement of the one point space $*$ in the $\sF$-model structure on $\Top^{G\times \Sigma}$ corresponding to the family $\sF$ described above.

\subsection{Operads}\label{sect:operads}

Given $n\ge 0$, let $\Sigma_n$ denote the symmetric group on $n$ letters,
where by convention $\Sigma_0= \Sigma_1$ is the trivial group. Let $(\M, S, \otimes)$ be a
a symmetric monoidal category and let $\M^{\Sigma_n}$ the category of objects of $\M$ which have a right $\Sigma_n$-action.

An \emph{operad} in $\M$ is a symmetric sequence $P=(P(n))_{n\in \N}$ of
objects (that is, each $P(n)$ is an object in $\M^{\Sigma_n}$) equipped with
an identity map $S\to P(1)$ and composition product maps
$$
P(n)\otimes P(k_1)\otimes\cdots\otimes P(k_n) \to P(k_1+\cdots+k_n)
$$
satisfying associativity, identity, and equivariance axioms (with respect to
$\Sigma_n$); see for instance~\cite{May_def_operads}. A \emph{morphism of
operads} is a morphism of the underlying collections that is compatible with
the unit maps and the composition product maps. We denote by $\Oper(\M)$ the
category of operads in $\M$.

A \emph{$P$-algebra} is an object $X$ of $\M$ together with an action of $P$
on $X$ given by maps $P(n)\otimes X^{\otimes n} \to X$, for every $n\ge 0$,
compatible with the symmetric group action, the unit of $P$, and subject to
the usual associativity relations. We denote by $\Alg_P(\M)$ the category of
$P$-algebras in $\M$.

We recall now the method by which cofibrancy is defined for an operad.
Let $(\M, S, \otimes)$ be a cofibrantly generated monoidal model category
and consider the category of collections in $\M$
$$
\Coll(\M) = \prod_{n\geq 0} \M^{\Sigma_n}.
$$
We can endow $\Coll(\M)$ with a model structure via the product model
structure. There are several choices for model structures on $\M^{\Sigma_n}$
which will give different model structures on $\Coll(\M) $. We could use the projective
model structure, that is, a map $f$ is a fibration or a weak equivalence in $\M^{\Sigma_n}$ if
when we forget the $\Sigma_n$ action $f$ is a fibration or a weak equivalence
in~$\M$. Other possibility is to take the $\Sigma_n$-equivariant model structure, where
a map $f$ is a fibration or a weak equivalence in $\M^{\Sigma_n}$ if
$f^H$ is a fibration or a weak equivalence in~$\M$ for every $H\le \Sigma_n$, where $(-)^H$ is
the $H$-fixed points functor.

In any case, a~map $f$ in $\Coll(\M)$ is a weak equivalence, a fibration, or a cofibration if and only if
$f_n$ is a weak equivalence, a fibration or a cofibration in $\M^{\Sigma_n}$ for every $n$,
respectively. There is a free-forgetful adjunction
$$
\xymatrix{
F \colon\Coll(\M) \ar@<3pt>[r] & \ar@<1pt>[l] \Oper(\M)\colon U,
}
$$
where $U$ is the forgetful functor, and the left adjoint is the free operad
generated by a collection.

An operad $P$ is said to be \emph{underlying cofibrant} if it is cofibrant as a
collection after applying the forgetful functor, that is, if $P(n)$ is cofibrant
in $\M^{\Sigma_n}$ for every $n\ge 0$. When the projective model structure on $\M^{\Sigma_n}$ is used, $P$ is called \emph{$\Sigma$-cofibrant} if the map $I\to P$ satisfies the left lifting property with respect to all
trivial fibrations of collections, where $I$ denotes the initial
object in the category of operads, that is, $I(1)=S$ and $I(n)=\emptyset$ if $\ne 1$.
Observe that if the unit of the monoidal category $S$ is cofibrant, then every
$\Sigma$-cofibrant operad is underlying cofibrant.

Berger and Moerdijk considered the passage of a model structure from the category of collections to the category of operads via the
free-forgetful adjunction \cite{BM03}, where $\M^{\Sigma_n}$ is equipped with the projective model structure. Under certain hypotheses on $\M$, this transfer
endows the category of operads with a model structure so that the forgetful
functor creates weak equivalences and fibrations; see \cite[Theorem~3.1 and
Theorem~3.2]{BM03}.
Rezk considered in his thesis the same transfer in the case that $\M$ is the category of simplicial sets and $\M^{\Sigma_n}$ is equipped with the equivariant model structure; see \cite[Proposition~3.1.5 and Proposition 3.2.11]{rezk-phd}.

The existence of both model structures follow from the following transfer principle (\cite[Sections 2.5 and 2.6]{BM03}, \cite[Proposition 3.1.5]{rezk-phd}), which is based in Quillen's path-object argument~\cite[II~p.4.9]{Quillen}.
\begin{theorem}\label{thm:transfer}
Let $\M$ be a cofibrantly generated model category with $I$ and $J$ the set of generating cofibrations
and generating trivial cofibrations, respectively. Let $\mathcal{N}$ be a category with small colimits and
finite limits, and let $F\colon \M\rightleftarrows \mathcal{N}\colon U$ be an adjunction. Suppose that
\begin{itemize}
\item[{\rm (i)}] the left adjoint $F$ preserves small objects, and the domains of the maps in $F(I)$ (resp. $F(J)$) are small relative to $F(I)$-cell (resp. $F(J)$-cell)
\item[{\rm (ii)}] $\mathcal{N}$ has a fibrant replacement functor, i.e., there is a coaugmented functor $(E,\epsilon)$ on $\mathcal{N}$ such that
$UEX$ is fibrant in $\M$ and $U\epsilon_X$ is a weak equivalence in $\M$, for every $X$ in $\mathcal{N}$, and
\item[{\rm (iii)}] $\mathcal{N}$ has functorial path-objects for every $X$ such that $UX$ is fibrant in $\M$.
\end{itemize}
Then, there is a cofibrantly generated model structure on $\mathcal{N}$, where a map $f$ is a fibration or a weak equivalence if and only if $Uf$ is a weak equivalence of a fibration in $\M$, respectively. The
set of generating cofibrations and generating trivial cofibrations of $\mathcal{N}$ are$FI$ and $FJ$,
respectively.
\end{theorem}

\begin{remark} \label{remark:rezk-(iii)-comes-for-free}
If $\M$ is a simplicial model category, $\mathcal{N}$ a category enriched over simplicial sets and the adjunction is a simplicial adjunction, then condition (iii) is automatically fulfilled since $(-)^{\Delta[1]}$
works as a path-object functor; see~\cite[Proposition 3.1.5]{rezk-phd}.
\end{remark}
Even if we cannot transfer the model structure to the category of operads, we
know that a trivial fibration of operads should be a map which is a trivial
fibration when viewed as a collection. Thus, we may define an operad to be
\emph{cofibrant} if the map of operads $I \to P$ satisfies the left
lifting property (in the category of operads) with respect to all trivial fibrations of operads. In
particular, if $P$ is $\Sigma$-cofibrant, then the map $S\to P(1)$ is a
cofibration.

We will make use of the previous results in Sections~\ref{sec:model-operads}
and~\ref{sec:cof-rep-operads}.

\subsection{Equivariant Operads}

We now restrict attention to the model category of $G$-spaces, where $G$ is a compact Lie group. The reason for restricting to compact Lie groups is explained in \cite[Preface]{lewis-may-steinberger}: for larger classes of groups, the connection to representation theory is lost, even though aspects of the homotopy theory are possible. The definitions in this section first appeared in \cite[Ch.VII]{lewis-may-steinberger}.

\begin{defn}
A \emph{$G$-operad} in $\Top$ is an operad $P$ valued in $\Top^G$.  That is, $P$ consists of a sequence $(P(n))_{n\in\mathbb{N}}$ of $G\times \Sigma_n$\nobreakdash-spaces, with $G$ acting on the left and $\Sigma_n$ acting on the right, together with $G$-equivariant composition product maps. Furthermore, $G$ must fix the unit in $P(1)$.
\end{defn}

A morphism of $G$-operads $f:P\to Q$ is a sequence $(f_n\colon P(n)\to Q(n))_{n\in\mathbb{N}}$ of $G\times \Sigma_n$-equivariant maps such that it preserves the unit and is compatible with the composition product of the operads. We denote by $\Oper(\Top^G)$ the category of $G$-operads. The condition that $G$ must fix the unit is needed to ensure good behavior of $P$-algebras.

Following~\cite[Ch. IV, \S{1}]{lewis-may-steinberger}, we say that a principal $(G, \Sigma_n)$-bundle $X\to X/\Sigma_n$ is a principal $\Sigma_n$-bundle and a $G$-morphism such that $G$ acts on $X$ via $G$-bundle maps.

\begin{defn}
A $G$-operad $P$ is \emph{$\Sigma$-free} if all $P(n)$ are universal spaces of
principal $(G,\Sigma_n)$-bundles. It is an \emph{$E_\infty$-$G$-operad} if
these bundles are universal. It is \emph{cellular} if all $P(n)$ are
$(G\times \Sigma_n)$-CW complexes.
\end{defn}

We depart from \cite[Ch.VII]{lewis-may-steinberger} in our definition of an
algebra over an equivariant operad. Rather than requiring the twisted half
smash product, we let the operad act in the more modern way, in this case
using that $G$-spaces and $G$-spectra are both tensored over $G$-spaces.

\begin{defn}
Let $P$ be a $G$-operad. A \emph{$P$-algebra} in $\Top^G$ is a based
$G$-space $X$ together with $G$-maps $P(n)\times X^n \to X$ compatible with
the $\Sigma_n$-action and the operad structure maps.
\end{defn}

A \emph{map of $P$-algebras} is a map of $G$-spaces which is compatible with
the $P$-action. We denote by $\Alg_P(\Top^G)$ the category of $P$-algebras in $\Top^G$.

\subsection{Fixed-point model structures for $G$-spectra}

Moving now to fixed-point model structures on $G$-spectra (which will be required in Section \ref{sec:admissibility}), we follow
\cite{hovey-white} and define a \textit{$G$-spectrum} to be an orthogonal
spectrum with a $G$-action, that is, a sequence $X$ of pointed $G\times
O(n)$-spaces $X_n$ for $n\geq 0$, where $O(n)$ denotes the orthogonal
group of dimension $n$, with associative and unital $G\times O(n)\times O(m)$-equivariant
structure maps $S^n\wedge X_m \to X_{n+m}$ (such $X$ are called $G$-orthogonal
sequences). When $G = \{e\}$, a $G$-spectrum
is an orthogonal spectrum. The category of $G$-spectra is closed symmetric
monoidal, because a $G$-spectrum $X$ is an $S$\nobreakdash-module in the category of
$G$-orthogonal sequences (here $S$ is the sphere spectrum, and a commutative
monoid in the usual way). The monoidal product on $G$-orthogonal sequences is
given by
\[
(X\otimes Y)_{n} = \bigvee_{p+q=n} O (n)_{+} \wedge_{O (p)\times O
(q)} (X_{p}\wedge Y_{q})
\]
with diagonal $G$-action.  The closed structure is given by
\[
\hom (X,Y)_{n} = \prod_{m\geq n} \Map_{O (m-n)} (X_{m-n}, Y_{m}),
\]
where $g\in O (n)$ acts on a map $f$ by acting on $f(x)\in Y_{m}$ using the
inclusion
\[
O (n) \subseteq O (m-n)\times O (n) \xrightarrow{} O (m).
\]
The enrichment over topological spaces is given by
\[
\Map (X,Y) = \prod_{n} \Map_{G\times O (n)} (X_{n}, Y_{n}).
\]

Following \cite{hovey-white}, let $\mathcal{U}$ denote a complete
$G$-universe, let $V$ be an $n$-dimension $G$-representation in
$\mathcal{U}$, and let $\ev_V$ be the functor from $G$-spectra to $\Top_*^G$
which takes a spectrum $X$ to the space
\[
X (V) = O (\R^{n},V)_{+} \wedge_{O (n)} X_{n}.
\]
The left adjoint to $\ev_V$ is $F_V$, defined as
\[
(F_{V}K)_{n+k} = O (n+k)_{+} \wedge_{O (k)\times O (n)} (S^{k}\wedge
(O (V, \R^{n})_{+}\wedge K)),
\]

\begin{proposition}
For any family of subgroups $\sH$ of $G$, we may endow the category of $G$\nobreakdash-spectra with a proper, cellular, monoidal, topological model structure denoted $\Sp^\sH$, which is
the stabilization of the model structure $\Top^\sH$, where weak equivalences and fibrations are maps $f$ such that $f^H$ is a weak equivalence or fibration for all $H \in \sH$.
\end{proposition}

\begin{proof} The proof proceeds just as in \cite[Theorems 3.3 and 4.5]{hovey-white}, where now the generating cofibrations are the maps $F_{V}i$ for $i$ a generating cofibration
\[
(G/H \times S^{n-1})_{+}\xrightarrow{}(G/H \times D^{n})_{+}
\]
of $\Top^\sH$.  The generating trivial cofibrations are the maps $F_{V}j$ for
$j$ a generating trivial cofibration
\[
(G/H \times D^{n})_{+} \xrightarrow{} (G/H\times D^{n}\times
D^{1})_{+}
\]
of $\Top^\sH$. This model category $\Sp^\sH$ is enriched over $\Top$ by
considering $\Map_{\Sp^\sH} (X,Y)$ as a subspace of $\Map (X,Y)$ consisting
of maps of orthogonal spectra.

The proof of cellularity (in the sense of \cite[Section 12.1]{hirschhorn}) is technical, but
is necessary in order for left Bousfield localizations to exist for general sets of maps. We delay it until Proposition~\ref{prop:G-spectra-cellular}.
\end{proof}

\begin{remark}
If a combinatorial model category is used for $\Top^G_+$, e.g., $\Delta$-generated spaces or simplicial sets, then the model structure above is also combinatorial. This is discussed in \cite[Section 8]{white-localization}.
\end{remark}

\section{Family Model Structures on $G$-Operads} \label{sec:model-operads}

In this section we will use the transfer principle (Theorem~\ref{thm:transfer}) to produce fixed-point model structures on the category of $G$-operads corresponding to sequences $\sF = (\sF_n)$ of families of subgroups of $G\times \Sigma_n$. Theorem \ref{thm:family-oper} generalizes \cite[Proposition 3.2.11]{rezk-phd}, both to the equivariant context and to the context of topological spaces. A similar result, regarding so-called genuine equivariant operads (built using $G$-equivariant trees), is contained in \cite{bonventre-pereira}. In Section \ref{sec:cof-rep-operads} we will use the model structures of Theorem \ref{thm:family-oper} to discuss the operads $E_\infty^\sF$, the study of which forms the heart of this paper.

\begin{theorem} \label{thm:family-oper}
Let $\sF=\{\sF_n\}_{n\ge 0}$ where each $\sF_n$ is a family of subgroups of $G\times \Sigma_n$. The category of $G$-operads inherits a transferred model structure, called the $\sF$-model structure, in which a map of $G$-operads $f\colon P\to Q$ is a weak equivalence or a fibration, if for every $n\ge 0$, the map $f(n)\colon P(n)\to Q(n)$ is a weak equivalences or a fibration in $\Top^{G\times \Sigma_n}$ with the $\sF_n$-model structure, respectively. More explicitly, $f$ is a weak equivalence or a fibration if $f(n)^H$ is a weak equivalence or a fibration in $\Top$, for every $H\in\sF_n$.
\end{theorem}

Our proof of this theorem works with either pointed or unpointed spaces. 
In order to apply the transfer principle, we start with the category of $G$-collections. We define a \textit{$G$-collection} to be a collection of $G$-spaces, that is, a sequence of spaces $(C_n)_{n\in \N}$ where $C_n$ is a $G\times \Sigma_n$-space. Let $\Coll_G=\prod_{n\ge 0}(\Top^{G})^{\Sigma_n}$ denote the category of $G$-collections. We have seen that the action of $G$ needed in the definition of a $G$-operad may be encoded internally at the level of collections and then passed from $\Coll_G$ to $Oper^G$ via the usual free operad functor $F$ \cite{BM03}. In other words, we do not need to consider actions of $G$ on the operad trees. The categorical algebra and the construction of the free operad functor on a collection are independent of the chosen model structure.

We define a model structure on $\Coll_G$ in the usual way, as the product model structure coming from the $\sF_n$ model structures on $\Top^{G\times\Sigma_n}$:

\begin{align*}
\Coll_G = \prod_{n\ge 0} (\Top^G)^{\Sigma_n} = \prod_{n\ge 0} \Top^{G\times \Sigma_n}.
\end{align*}
We are now prepared to prove the main theorem of this section.

\begin{proof}[Proof of Theorem \ref{thm:family-oper}]
Let $\sF_n$ be a family of subgroups of $G\times \Sigma_n$ for every $n\ge 0$. Let $\Coll_G$ be the category of $G$-collection with the model structure associated to the families $\sF_n$, as describe above.

We now apply Theorem~\ref{thm:transfer} to transfer the
model structure on $\Coll_G$ to a model structure on
$G$-operads via the free forgetful adjunction
$$
\xymatrix{
F \colon\Coll_G \ar@<3pt>[r] & \ar@<1pt>[l] \Oper^G\colon U.
}
$$
Since $\Top^G$ is a simplicial category, we just need to check conditions (i) and (ii) from Theorem~\ref{thm:transfer}, thanks to Remark \ref{remark:rezk-(iii)-comes-for-free}.

Condition (i) holds because all spaces are small relative to inclusions,  the generating (trivial) cofibrations $I$ (resp. $J$) of $G$-collections are inclusions, and the maps in $F(I)$-cell (resp. $F(J)$-cell) are built from $I$ (resp. $J$) via transfinite composition, pushouts, products with identity maps, and quotients by symmetric group actions (all of which preserve inclusions). Every object in $\Top^{G\times \Sigma_n}$ with the $\sF_n$-model structure is fibrant, so condition (ii) holds trivially if we use the identity functor as the functor $E$.
\end{proof}

\begin{remark}
In the model structure on $\Coll_G$ above, the domains of the generating (trivial) cofibrations are cofibrant. To see this, note that each $\Top^{G\times \Sigma_n}$ has cofibrant domains of the generating (trivial) cofibrations. It follows that the same is true for the $\sF$-model structure on $G$-operads.
\end{remark}

\begin{remark} \label{remark:simplicial-operads}
Theorem \ref{thm:family-oper} also holds if $sSet^G$ is used instead of $\Top^G$. Condition (i) holds automatically, since all simplicial sets are small. If $G$ is a discrete group, then the Quillen equivalence between topological spaces and
simplicial sets given by the geometric realization $|-|$ and singular functor $\Sing(-)$ induces a Quillen equivalence between $G\times \Sigma_n$-topological spaces and $G\times\Sigma_n$-simplicial sets, both with the $\sF_n$-model structure (which we also denote by $|-|$ and $\Sing(-)$). Moreover, these functors are lax monoidal, so applying the composite $\Sing(|-|)$ levelwise gives as a functor $E$ in $\Oper^G$ satisfying condition (ii). Alternatively, the $Ex^\infty$ functor may be used as a fibrant replacement.
\end{remark}

Using the model structure of Theorem \ref{thm:family-oper}, we can define the $E_\infty^\sF$-operads discussed in the introduction. To do so, recall that $\Com$ is the terminal $G$-operad, whose spaces $\Com(n) = \ast$ with a trivial $G$-action, for all $n$.

\begin{defn} \label{defn:E-infty-F-operad}
Let $\sF=(\sF_n)$ be a sequence of families $\sF_n$ of subgroups of $G\times \Sigma_n$. An \emph{$E_{\infty}^{\sF}$-operad} is a cofibrant replacement of $\Com$ in the $\sF$-model structure on $\Oper^G$.
\end{defn}

Observe that, by \cite[Section 2.6]{cort_ellis}, any $G$-operad $P$ that is cofibrant as a $G$-collection (with the $\sF$-model structure), and weakly equivalent to $\Com$, will have spaces $P(n)$ that are universal classifying spaces of the families $\sF_n$. We will use this observation in Section \ref{sec:cof-rep-operads} to prove the Blumberg-Hill conjecture. Note that not every cofibrant $G$-operad forgets to a cofibrant $G$-collection. Recall from Section~\ref{sect:operads} that an operad $P$ is called $\Sigma$-cofibrant if the map from the initial operad $I$ to $P$ forgets to a cofibration in the projective model structure on collections. This implies the operad is cofibrant as a collection (at least, if the monoidal unit is cofibrant). In order to distinguish our more general setting, we make the following definition.

\begin{defn}\label{def:F-cof}
Let $\sF=(\sF_n)$ be a sequence of families $\sF_n$ of subgroups of $G\times \Sigma_n$. A $G$-operad $P$ is called \emph{$\sF$-cofibrant} if the map from the initial operad $I$ to $P$ forgets to a cofibration in the $\sF$-model structure on $G$-collections.
\end{defn}

\begin{example} \label{ex:recovering-BM-model}
If the sequence $\sF$ has $\sF_n = 1 \times 1$ for all $n$, then the model structure of Theorem \ref{thm:family-oper} coincides with the model structure on $G$-operads obtained from applying \cite[Theorem 3.2]{BM03} to the base model category $\Top^\sH$ with the family model structure corresponding to the trivial family $\sH = 1$. The corresponding operad $E_\infty^\sF$ is non-equivariantly contractible (indeed, the $G$-action is free) and has a free $\Sigma_n$-action. If $G$ is trivial, then Theorem \ref{thm:family-oper} recovers the model structure of \cite[Section  3.3.2]{BM03}, where the base is $\Top$. For any family $\sH$ of subgroups of $G$, there is an operad $E_\infty^\sH$ defined from the sequence $\sF_n = H \times 1$ for all $H \in \sH$. These operads interpolate between a trivial $G$-action (when $\sH$ is all $H < G$) and a free $G$-action (when $\sH$ is trivial).
\end{example}

\begin{example}  \label{ex:recover-E-infty}
If the sequence $\sF$ has $\sF_n = \{H \times 1 \;|\; H < G\}$, i.e., we take the family $\sH$ of all subgroups of $G$ in Example \ref{ex:recovering-BM-model}, then the model structure of Theorem \ref{thm:family-oper} coincides with the model structure on $G$-operads obtained from applying \cite[Theorem 3.2]{BM03} to the base model category $\Top^G$, with its usual fixed-point model structure. We denote the corresponding operad $E_\infty^\sF$ as $E_\infty^G$. Observe that $E_\infty^G(n)$ is equivariantly contractible and has a free $\Sigma_n$-action. This operad can act in any $G$-topological model category, via the enrichment in $\Top^G$. In $G$-spectra, its algebras are equivalent to algebras over an $E_\infty$-operad as in Example \ref{ex:recovering-BM-model}, because an equivariant map from $E\Sigma_n$ (with a trivial $G$-action) into the $G$-space $\hom_{\Top^G}(X^n,X)$ must land in the $G$-fixed points, i.e., in the topological enrichment $\hom_{\Top}(X^n,X)$.
\end{example}

\begin{example} \label{ex:complete-N-infty-operads}
If, for all $n$, $\sF_n$ consists of all graph subgroups $\Gamma_\rho$ where $H<G$ and $\rho: H\to \Sigma_n$, then the corresponding $E_\infty^\sF$-operad is an $E_\infty$-$G$-operad \cite[Definition VII.1.2]{lewis-may-steinberger}. Its $n^{th}$-space is $E_G(\Sigma_n)$, the universal space of the universal
$G$-equivariant principle $\Sigma_n$-bundle. If $G$ is finite, then $E_\infty^\sF$-algebras in $G$-spectra are Quillen equivalent to strictly commutative $G$-ring spectra (this is a consequence of Theorem \ref{thm:family-rectification} below), i.e., $E_\infty^\sF$-algebras have all multiplicative norms.
\end{example}

\begin{example} \label{ex:recovering-rezk-model}
If $G$ is the trivial group, and if the sequence $\sF_n = \{1 \times K \;|\; K < \Sigma_n\}$, then the model structure of Theorem \ref{thm:family-oper} recovers the Rezk model structure on operads \cite[Proposition 3.2.11]{rezk-phd}), using Remark \ref{remark:simplicial-operads}, and extends it to compactly generated spaces. We see that working with families of subgroups that intersect $\Sigma_n$ non-trivially is the same as allowing non-free $\Sigma_n$-actions. When we work simplicially, and when $\sF$ is the sequence $\sF_n$ of families consisting of all subgroups of $G\times \Sigma_n$, then all objects of $\Coll_G$ are cofibrant with respect to the $\sF$-model structure \cite[Proposition 3.1.9]{rezk-phd}. It follows that our theory recovers the $\Com$ operad in this case.
\end{example}

Our last example discusses certain ``partial multiplications" encoded on algebras over $E_\infty^\sF$ operads that do not contain the subgroup $G \times 1$. This example is based on an observation Mike Hill made to the second author in 2014.

\begin{example} \label{ex:restrictions}
Let $\sH$ be a family of subgroups of $G$. If a sequence $\sF_n$ contains all subgroups of the form $H \times 1$ for $H \in \sH$, then any $E_\infty^\sF$-algebra $X$ will have multiplications on $\res_H(X)$ for all $H\in \sH$. To see this, we first recall why $X$ has a multiplication, when $G\times 1\in \sF$. Rewriting $X^{\wedge n}$ as $(\Sigma_n)_+ \wedge_{\Sigma_n} X^{\wedge n} \cong ((G\times \Sigma_n)/(G\times 1)_+ \wedge_{\Sigma_n} X^{\wedge n}$. From here, the universal property of $E_\infty^\sF$ guarantees a $G\times \Sigma_n$-equivariant map $(G\times \Sigma_n)/(G\times 1) \to (E_\infty^\sF(n))_+$, since the isotropy group of $(G\times \Sigma_n)/(G\times 1)$ is in $\sF$ \cite[Definition 1.8]{luck-survey-classifying-families}. Composing with the operad-algebra structure map $(E_\infty^\sF (n))_+ \wedge_{\Sigma_n} X^{\wedge n} \to X$ provides the desired multiplication $X^{\wedge n} \to X$. However, if $G\times 1$ is not in $\sF_n$, then we will only have maps $(G\times \Sigma_n)/(H\times 1) \to E_\infty^\sH (n)$ for $H\in \sH$. This corresponds to a multiplication $\res_H(X)^{\wedge n} \to \res_H(X)$. From this point of view, the theory of $E_\infty^\sF$ operads can be viewed as an enlargement of the theory of $N_\infty$-operads to allow for restricted multiplications and non-free $\Sigma_n$-actions.
\end{example}

\section{Proving $N_\infty$-operads exist}
\label{sec:cof-rep-operads}

In this section we prove a conjecture of Blumberg and Hill. Specifically, for any sequence $\sF = (\sF_n)$ of families of subgroups of $G\times \Sigma_n$, satisfying a certain condition relating the families as $n$ varies (a condition required in order to have operad composition maps), we prove that there is an associated $N_\infty$-operad $P$ realizing the family, i.e., such that $P(n)$ is a universal classifying space for the family $\sF_n$. Throughout this section we will restrict to the types of sequences of families of subgroups considered by Blumberg and Hill, i.e., all $\sF_n$ will contain all subgroups of the form $H \times 1$ (for all closed $H < G$) and will consist of graph subgroups $\Gamma_\rho = \{(h,\rho(h))\}$ for various $\rho: H \to \Sigma_n$. Throughout the section, $G$ is a compact Lie group. The notion of equivariant $N_{\infty}$-operad was introduced in~\cite[Definition 3.3]{blumberg-hill}.

\begin{definition} \label{defn:N-infty}
A $G$-operad $P$ is called $N_\infty$ if the action of $\Sigma_n$ on $P(n)$ is free, and if $P(n)$ is a universal space for a family $\sN_n$ of subgroups of $G\times \Sigma_n$ such that $H\times 1 \in \sN_n$ for all closed subgroups $H < G$ and for all $n$.
\end{definition}

Recall from Section \ref{sec:prelim} that this means
\[
P(n)^K \simeq
\begin{cases}
      \emptyset & \text{ if } K\not\in \sN_n(P), \\
      \ast & \text{ if } K\in \sN_n(P).
   \end{cases}
\]

\begin{remark}
Blumberg and Hill originally required $P(0)$ to be $G$-contractible, but this follows from the condition of being a universal classifying space for a family $\sN_0$ of subgroups of $G\times \Sigma_0 \cong G$ containing all closed $H < G$. Similarly, all of the spaces $P(n)$ are contractible in $\Top^G$, since $\sN_n$ contains all $H\times 1$.
\end{remark}

\begin{remark}
Requiring the $\Sigma_n$-action on $P(n)$ to be free is the same as requiring the fixed points $P(n)^{1\times K} \cong \emptyset$ for all non-trivial $K < \Sigma_n$, i.e., the subgroup $1\times K$ cannot be in the family $\sN_n$. Furthermore, for any $\Gamma \in \sN_n$, $\Gamma \cap (1\times \Sigma_n) = \{1 \times 1\}$, because any element in $\Gamma$ must fix something in $P(n)$, since $P(n)^{\Gamma} \simeq \ast$, but any element in $1\times \Sigma_n$, other than the identity, cannot fix anything in $P(n)$. This implies that the only subgroups that can occur in $\sN_n$ are graphs of group homomorphisms $\rho: H\to \Sigma_n$, where $H<G$, i.e., $\Gamma$ must have the form $\Gamma_{H,\rho} = \{(h,\rho(h)) \;|\; h\in H < G, \rho: H\to \Sigma_n\}$.
\end{remark}

Based on this remark, we see that the only variable distinguishing $N_\infty$-operads is which $\rho$ are allowed (since $N_\infty$-operads require that all subgroups $H < G$ are considered). If the only $\rho$ allowed is $\rho(h) = e$ for all $h$, then we denote the resulting operad $E_\infty^G$. Its algebras do not have any non-trivial norm maps.

If all group homomorphisms $\rho$ are allowed, then the resulting operad is an $E_\infty$-$G$-operad \cite[Definition VII.1.2]{lewis-may-steinberger}, since the spaces $P(n)$ are universal spaces for principal $(G,\Sigma_n)$-bundles. Algebras over these operads have all possible norm maps, and there is a Quillen equivalence between such algebras (in $G$-spectra, when the $P(n)$ are assumed to be $(G\times \Sigma_n$)-CW complexes) and strictly commutative $G$-spectra; see Theorem \ref{thm:family-rectification}. Blumberg and Hill call these $G$-operads \textit{complete}. As the word ``genuine'' is already overused in equivariant homotopy theory, we think of these as ``strong'' $E_\infty$-operads, and we think of $E_\infty^G$ as a ``weak'' $E_\infty$-operad, but we will stick to the terminology of \cite{blumberg-hill}.
The notion of complete $N_{\infty}$-operad appears in~\cite[Section 3.1]{blumberg-hill}.
\begin{definition}
An $N_\infty$-operad $P$ is called \textit{complete} if the family $\sN_n$ corresponding to $P(n)$ is precisely the set of graph subgroups $\Gamma_{H,\rho} = \{(h,\rho(h)) \;|\; h\in H < G, \rho: H\to \Sigma_n\}$, where $H$ is any subgroup of $G$, and $\rho: H\to \Sigma_n$ is any group homomorphism.
\end{definition}

Since every $N_\infty$-operad gives rise to an associated sequence $\sN = (\sN_n)$ of families of subgroups, it is natural to ask if every sequence of families has an associated $N_\infty$-operad. The following example shows that this is not the case.

\begin{example} \label{ex:no-such-operad}
Consider a sequence of families $\sF = (\sF_i)$ and a fixed $n$, such that for each $k < n$, $\sF_{k} = \{\Gamma_{H,\rho} \}$ is the family of all graph subgroups $\Gamma_{H,\rho} < G\times \Sigma_k$, for each $H < G$ and each $\rho:H\to \Sigma_k$; and for $k\ge n$, $\sF_k = \{H \times 1\}$ consists only of graph subgroups of the trivial $\rho$. Then there cannot be a $G$-operad $P$ whose spaces $P(m)$ are universal spaces for these $\sF_m$, because any operad composition map must be $G$-equivariant, and taking fixed points will result in a map from a contractible (but non-empty) space to the empty set. For concreteness, suppose $n = n_1+n_2$, and consider $\gamma: P(2) \times P(n_1) \times P(n_2) \to P(n)$. This map is $G \times (\{e\} \times \Sigma_{n_1} \times \Sigma_{n_2})$-equivariant. Let $H < G$ and let $\rho_i: H\to \Sigma_{n_i}$ be group homomorphisms. Let $\Gamma_i < G \times \Sigma_{n_i}$ be the graph subgroups of $\rho_i$. Define $\rho: H\to \Sigma_n$ to be $\rho_1 \coprod \rho_2$, using the block inclusions of $\Sigma_{n_i}$ into $\Sigma_n$,
and let $\Gamma < G\times \Sigma_n$ be the graph subgroup of $\rho$. Then the domain of $\gamma^\Gamma$ will be $P(2) \times P(n_1)^{\Gamma_1} \times P(n_2)^{\Gamma_2} \cong \ast$, but the codomain is $P(n)^\Gamma \cong \emptyset$, contradicting the existence of the composition map $\gamma$.
\end{example}

With this example in mind, we now state the key condition relating the $\sF_n$ that will guarantee operad composition maps can exist. We then prove this condition is necessary and sufficient for the existence of $N_\infty$-operads relative to such $\sF$. The bulk of the work consists of showing that, for \textit{realizable} sequences $\sF$, cofibrant operads in the $\sF$-model structure on $G$-operads (e.g., $E_\infty^\sF$) forget to cofibrant collections in the $\sF$-model structure on $\Coll_G$ (meaning, the spaces of the operad are universal spaces for the families $\sF_n$). It is important to note that the cofibrant replacement of Com exists in any $\sF$-model structure on $G$-operads, even in situations like Example \ref{ex:no-such-operad} where there is no operad $P$ whose spaces are universal classifying spaces for the families $\sF_n$. This demonstrates that Proposition~\ref{prop:cof-operads-forget} is false without the realizability hypothesis, i.e., the property of cofibrant operads forgetting to cofibrant sequences does not come for free. One cannot side-step this failure by taking the $\sF$-cofibrant replacement of Com in $\Coll_G$, because a cofibrant replacement in the category of $G$-collections need not be an operad.

\begin{definition} \label{defn:realizable}
A sequence $\sF = (\sF_n)$ of families of subgroups of $G\times \Sigma_n$ is \textit{realizable} if, for each decomposition $n = n_1 + \dots + n_k$, the following containment is satisfied
\[
\sF_k \wr (\sF_{n_1} \times \dots \times \sF_{n_k}) \subset \sF_n,
\]
where the symbol on the left denotes a set of subgroups of $G\times \Sigma_n$ defined as follows. For every $H < G$, every $\rho_{n_i}: H \to \Sigma_{n_i}$ allowed by the family $\sF_{n_i}$, consider all block homomorphisms of the form $\rho_{n_1} \coprod \dots \coprod \rho_{n_k}: H\to \Sigma_n$. Then, for every $\rho_k: H\to \Sigma_k$ allowed by the family $\sF_k$,
consider all group homomorphisms $\rho: H \to \Sigma_n$  of the form $\rho_{n_1} \coprod \dots \coprod \rho_{n_k}$ twisted by $\rho_k(H)$. Then, $\sF_k \wr (\sF_{n_1} \times \dots \times \sF_{n_k})$ is defined as the set of graph subgroups $\Gamma_\rho$ of the resulting $\rho$.
\end{definition}

\begin{theorem} \label{thm:N-infty-operads-exist}
A sequence $\sF = (\sF_n)$ is realizable if and only if there is an $N_\infty$-operad $P$ such that $P(n)$ is a universal classifying space for the family $\sF_n$.
\end{theorem}

We will prove this theorem in a moment. First, to help the reader make sense of realizable sequences, we include a remark showing that realizable sequences satisfy the three closure properties of indexing systems discussed in \cite{blumberg-hill}. This can of course be deduced from Theorem \ref{thm:N-infty-operads-exist}, but in this remark we demonstrate it directly.

\begin{remark}\leavevmode
\begin{enumerate}
\item The coefficient system associated to a realizable sequence \cite[Definition 4.5]{blumberg-hill} (which does not require that the sequence underlies an operad) is closed under coproduct. Given $\rho_{n_1}: H\to \Sigma_{n_1}$ and $\rho_{n_2}: H\to \Sigma_{n_2}$, the coproduct graph subgroup is given by $\rho: H\to \Sigma_{n_1+n_2}$ taking $h$ to $\rho_{n_1}(h) \coprod \rho_{n_2}(h)$. This $\rho$ is an example of the sort considered in Definition \ref{defn:realizable}, where we take the trivial action $\rho_2: H\to \Sigma_2$. That that coproduct graph is in the coefficient system can be verified just as in the proof of \cite[Lemma 4.10]{blumberg-hill}, where $\Gamma_\rho$ fixed points are decomposed into $\Gamma_{\rho_{n_1}}$ and $\Gamma_{\rho_{n_2}}$ fixed points. The containment required by Definition \ref{defn:realizable} implies that $\Gamma_\rho$ fixed points are contractible as required.

\item The coefficient system associated to a realizable sequence is closed under products. Let $S$ be an admissible $H$-set of cardinality $k$ (meaning, $\Gamma_{\rho_k} \in \sF_k$ where $\rho_k$ encodes the $H$-action on $S$), and let $T$ be an admissible $H$-set of cardinality $q$. Here, one mimics \cite[Lemma 4.11]{blumberg-hill}, taking $n_1 = \dots = n_k$ in Definition \ref{defn:realizable}, and taking all $\rho_{n_i}$ to be the homomorphism associated to the $H$-set $T$. Then, the subgroup $\Gamma_{S x T}$ of \cite[Lemma 4.11]{blumberg-hill} is precisely the same as $\Gamma_\rho$ from Definition~\ref{defn:realizable}.

\item The coefficient system associated to a realizable sequence is closed under self-induction \cite[Definition 3.14]{blumberg-hill}. Let $K < H < G$, and assume $H/K$ is an admissible $H$-set. Fix an admissible $K$-set $T$ given by $\tau_q: K \to \Sigma_q$. In order to show that $H\times_K T$ is an admissible $H$-set, \cite[Lemma 4.12]{blumberg-hill} defines $k$ homomorphisms $H \to K$ via a complete set of coset representatives of $H/K$ (here $k  = |H/K|$). Blumberg and Hill then build a subgroup $\Ind(g)$ from $\rho_k: H \to \Sigma_{k}$ (using that $H/K$ is an admissible $H$-set) and from the compositions $\rho_{q}^i: H \to K \to \Sigma_{q}$ for $1\leq i \leq k$. This $\Ind(g)$ is precisely the subgroup $\Gamma_\rho$ of Definition \ref{defn:realizable}, where we take $n_1 = \dots = n_k = q$ and each $\rho_{n_i} = \rho_q^i$. Thus, the admissibility of $\Ind(g)$ follows from the containment in Definition \ref{defn:realizable}, once the $\Gamma_\rho$ fixed points are decomposed as in \cite[Lemma 4.12]{blumberg-hill}.
\end{enumerate}
\end{remark}

Now we are ready to prove Theorem \ref{thm:N-infty-operads-exist}. Our proof relies on Proposition \ref{prop:cof-operads-forget}, to be proven below.

\begin{proof}[Proof of Theorem \ref{thm:N-infty-operads-exist}]
Suppose $\sF$ is realizable. We construct $P$ as the cofibrant replacement of the Com operad in the $\sF$-model structure on $G$-operads. The $G$-operad Com has $\Com(n) = \ast$ with a trivial action, for all $n$.
Then, using the realizability hypothesis on $\sF$, Proposition \ref{prop:cof-operads-forget} proves that $P$ is $\sF$-cofibrant, which implies $P(n)$ is cofibrant in the $\sF_n$-model structure on $\Top^{G\times \Sigma_n}$. As explained in~\cite[Section 2.6]{cort_ellis} and \cite[Section 2]{fausk-equivariant-pro-spectra}, $E_{\sF_n}G$ is a cofibrant replacement of the one point space $*$ in this model structure, so $P(n)$ is weakly equivalent to the $G\times \Sigma_n$-CW complex $E_{\sF}G$ discussed in Section \ref{sec:prelim}, and has the same fixed point property, as required.

Conversely, suppose $P$ is an $N_\infty$-operad and let $\Gamma_\rho \in \sF_k \wr (\sF_{n_1} \times \dots \times \sF_{n_k})$. The composition map $\gamma: P(k) \times P(n_1) \times \dots \times P(n_k) \to P(n)$ is $G \times (\rho_k(H) \times \Sigma_{n_1} \times \dots \times \Sigma_{n_k})$-equivariant, since $\rho_k(H)$ acts by permuting blocks of the same size. Upon taking fixed points we obtain $\gamma^{\Gamma_\rho}: P(k)^{\Gamma_{\rho_k}} \times (P(n_1)^{\Gamma_{n_1}} \times \dots \times P(n_k)^{\Gamma_{n_k}})^{\rho_k(H)} \to P(n)^{\Gamma_\rho}$, following the model of the proof of \cite[Lemma 4.10 and Lemma 4.12]{blumberg-hill}. By construction, the left hand side is contractible: since taking $\rho_k(H)$-fixed points identifies certain copies of the $P(n_i)^{\Gamma_{n_i}}$, the left hand side is a product of $P(k)^{\Gamma_{\rho_k}}$ with various products of $P(n_i)^{\Gamma_{n_i}}$, which are all contractible. Hence, the right hand side cannot be empty, or it would contradict the existence of $\gamma$. It follows that $\Gamma_\rho \in \sF_n$.
\end{proof}

Our proof of Proposition \ref{prop:cof-operads-forget} proceeds following the model of \cite[Section 5]{BM03}, with the correction from  \cite[Lemma 3.1]{BM09}. We encourage the reader to proceed with a copy of \cite{BM03} on hand. The proof proceeds via a careful analysis of the free operad functor $F: \prod_n Top^{G\times \Sigma_n} \to Oper(Top^G)$. This is precisely the same functor as in \cite{BM03}, so the categorical algebra in \cite[Section 5.8]{BM03} (with the correction from \cite{BM09}) works in precisely the same way in our setting (the only difference is the model category structure, which does not enter until~\cite[Lemma 5.9]{BM03}). As in \cite[Corollary 5.2]{BM03}, it is completely formal to reduce the problem of a cofibration of operads forgetting to a cofibration of collections to the situation of a single cellular extension.
Thus, we must only prove the analogue of \cite[Proposition 5.1]{BM03} in our setting, namely: if $P$ is an $\sF$-cofibrant $G$-operad and $u$ is a cofibration in the $\sF$-model structure on $G$-collections, then a cellular extension $P \to P[u]$, defined as the pushout of $F(u)$ in $G$-operads, is an $\sF$-cofibration of operads.

Recall from Definition~\ref{def:F-cof} that a $G$-operad $P$ is $\sF$-cofibrant if the map from the initial operad $I$ forgets to a cofibration in the $\sF$-model structure on $G$-collections. Similarly, we call a map of operads an $\sF$-cofibration if it forgets to a cofibration in the $\sF$-model structure. Note that $I(n)$ is the initial object in $\Top^{G\times \Sigma_n}$ for all $n\neq 1$ and $I(1)$ is the unit of $\Top^{G\times \Sigma_1}\cong \Top^G$. Hence, each $I(n)$ is cofibrant in the $\sF_n$-model structure on $\Top^{G\times \Sigma_n}$. It follows that, if $P$ is $\sF$-cofibrant, then every space $P(n)$ is cofibrant in the $\sF_n$-model structure on $\Top^{G\times \Sigma_n}$. This observation is used in Theorem \ref{thm:N-infty-operads-exist}, to deduce that $P(n)$ is a universal classifying space for the family $\sF_n$, for the operads $P$ constructed in the theorem.

The proof by Berger--Moerdijk hinges on \cite[Lemma 5.9]{BM03}, and so most of our work will be to establish the $\sF$-model structure version of this lemma. However, since we work with family model structures, rather than projective model structures, we will need to prove that the latching maps are cofibrations in a model structure on $\Top^{G\times \Aut(T,c)}$ for all trees $T$ with colored vertices $c$. We now define this model structure, induced from the family model structure $\Top^{\sF_n}$ where $n$ is the number of leaves of $T$.

For every tree $T$ with colored vertices $c$, let $\alpha_{T,c}: \Aut(T,c) \to \Sigma_n$ be induced by the action of $\Aut(T,c)$ on the leaves, and let $\sF_n^{T,c} := (1 \times \alpha_{T,c})^{-1}(\sF_n)$ be a family of subgroups of $G \times \Aut(T,c)$. It is easy to verify that this family contains the identity subgroup, is closed under conjugation, and is closed under subgroups (or satisfies the Illman condition in the case of compact Lie $G$). Hence, there is a family model structure on $\Top^{G\times \Aut(T,c)}$, and it satisfies the pushout product axiom. We need the following lemma, which will be applied to the map $1\times \alpha: G\times \Aut(T,c)\to G\times \Sigma_n$, to prove that $-\otimes_{\Aut(T,c)} S[\Sigma_n]$ is a left Quillen functor from $\Top^{G\times \Aut(T,c)}$ to $\Top^{G\times \Sigma_n}$.

\begin{lemma}\label{lemma:aut-family-cof}
Let $\alpha: G_0 \to G_1$ be a group homomorphism, and $\sF_0$, $\sF_1$ be families of subgroups of $G_0$ and $G_1$, respectively, such that for every $H_0\in\sF_0$ we have that $\alpha(H_0)\in \sF_1$. Let $\cat{C}$ be a model category such that the family model structures $\cat{C}^{G_0}_{\sF_0}$ and $\cat{C}^{G_1}_{\sF_1}$ exist. Then the adjunction $\alpha_!: \cat{C}^{G_0}_{\sF_0} \leftrightarrows \cat{C}^{G_1}_{\sF_1}:\alpha^*$ is a Quillen pair.
\end{lemma}

\begin{proof}
For every $H_0 \in \sF_0$, we have $X^{\alpha(H_0)} \cong (\alpha^*(X))^{H_0}$ for every $X \in \cat{C}^{G_1}$. It is enough to see that $\alpha^*$ preserves weak equivalences and fibrations. Let $f$ be a weak equivalence or fibration in $\cat{C}^{G_1}$. Then $\alpha^*(f)^{H_0} \cong f^{\alpha(H_0)}$, which is a weak equivalence or fibration in $\cat{C}$ since $\alpha(H_0) \in \sF_1$ by assumption. The conclusion follows.
\end{proof}

We are finally ready to prove Proposition \ref{prop:cof-operads-forget}. A similar result appears in \cite{bonventre-pereira}, where $\sF$ is required to be a ``weak indexing system'' (a condition related to the behavior of automorphisms of $G$-equivariant trees), rather than a realizable sequence.

\begin{proposition} \label{prop:cof-operads-forget}
Assume $\sF$ is a realizable sequence. Then any $G$-operad $P$ that is cofibrant in the $\sF$-model structure on $G$-operads is $\sF$-cofibrant.
\end{proposition}

\begin{proof}
The discussion above reduces us to proving that, for every $\sF$-cofibrant $G$-operad $P$, every generating cofibration $u:K\to L$ in the $\sF$-model structure on $G$-collections, and every attaching map $K\to U(P)$, that the cellular extension $P \to P[u]$, defined by the following pushout in the category of $G$-operads, is an $\sF$-cofibration:

\[
\xymatrix{
F(K) \ar[r] \ar[d] \po & F(L) \ar[d] \\
P \ar[r] & P[u].}
\]

We proceed as in \cite{BM03}. First, the categorical algebra in \cite{BM03} is independent of the choice of model structure, and of the base category where operads are taken ($\Top^G$ in our case). Thus, the filtration of $P\to P[u]$ from \cite[Section 5.11]{BM03} still holds, with the correction from \cite[Lemma 3.1]{BM09}. This means $P\to P[u]$ is a sequential colimit of maps of operads $F_{k-1}\to F_k$ where, for all $n$, level $n$ is defined by the following pushout (whose notation will be defined below) in the category of $G\times \Sigma_n$-spaces:

\[
\xymatrix{
\coprod \limits_{(T,c)\in A_k(n)} \ulu^{*}(T,c) \otimes_{\Aut(T,c)} S[\Sigma_n] \ar[r] \ar[d] & \coprod \limits_{(T,c)\in A_k(n)}  \ulu(T,c) \otimes_{\Aut(T,c)} S[\Sigma_n] \ar[d] \\
F_{k-1}(n) \ar[r] & F_k(n).
}
\]

We must prove that the maps $F_{k-1}(n) \to F_k(n)$ are cofibrations in the $\sF_n$-model structure for all $k$. To do so, we prove that the maps $\ulu^{*}(T,c) \otimes_{\Aut(T,c)} S[\Sigma_n] \to \ulu(T,c) \otimes_{\Aut(T,c)} S[\Sigma_n]$ are cofibrations in the $\sF_n$-model structure.

We now define the notation, following \cite{BM03}. Let $A_k(n)$ denote the isomorphism classes of admissible coloured trees with $n$ inputs and $k$ vertices, which are either colored or unitary. Every tree $(T,c)$ in $A_k(n)$ with a root of valence $v$ is a grafting $t_v(T_1,\dots,T_v)$ of trees $T_i$ with $n_i$ leaves, and colors induced from $(T,c)$. Let $\ulu(T,c)$ be inductively defined as $K(n) \otimes \ulv^*(T,c)$, where $\ulv^*(T,c) = \bigcup_{i=1}^v \limits \ulu(T_1,c_1) \otimes \dots \ulu^*(T_i,c_i) \otimes \dots \otimes \ulu(T_v,c_v)$, if the root of $T$ is uncolored and not unitary. Let $\ulu(T,c)$ be $K(1) \otimes \ulv^*(T,c) \cup I \otimes (\ulu(T_1,c_1) \otimes \dots \otimes \ulu(T_v,c_v))$ if the root is uncolored and unitary. Finally, let $\ulu(T,c)$ be $L(n) \otimes \ulv^*(T,c) \cup K(n)\otimes \ulu(T_1,c_1) \otimes \dots \otimes \ulu(T_v,c_v)$ if the root of $T$ is colored. Let $\ulu^*(T,c) := \colim_{c' \subsetneq c} \limits \ulu(T,c')$. The map $\ulu^*(T,c)\to \ulu(T,c)$ in the pushout above is the colimit (latching) map.

By Lemma \ref{lemma:aut-family-cof}, we must only show that the latching map is a cofibration in the family $\sF_n^{T,c}$. This will proceed by induction, using the decomposition (from \cite[Lemma 5.9]{BM03}) $\Aut(T,c) \cong (\Aut(T_1,c_1)^{m_1} \times \dots \times (\Aut(T_r,c_r)^{m_r}) \rtimes (\Sigma_{m_1} \times \dots \times \Sigma_{m_r})$, where $T$ is the grafting of $T_1^1,\dots,T_1^{m_1},T_2^1,\dots,T_2^{m_2},\dots,T_r^1,\dots,T_r^{m_r}$, $v = m_1 + \dots + m_r$ is the degree of the root, and the action of $(\Sigma_{m_1} \times \dots \times \Sigma_{m_r})$ permutes isomorphic trees. Recall that $\alpha: \Aut(T,c)\to \Sigma_n$ is the induced action on the leaves of $T$. The realizability hypothesis on $\sF$ implies that $\alpha^{-1}(\sF_v \wr (\sF_{n_1} \times \dots \times \sF_{n_v})) \subset \alpha^{-1}(\sF_n)$, i.e., all of the cells that the latching map builds in (induced up from the various $\Top^{G\times \Aut(T_i,c_i)}$ model structures) are contained in the cells used in the generating cofibrations of $\Top^{G\times \Aut(T,c)}$.

We now carry out the induction in each of the cases from \cite[Lemma 3.1]{BM09}. If the root is uncolored and not unitary, the latching map is $K(n) \otimes (A\to B)$ where $B = \ulu(T_1,c_1) \otimes \dots \otimes \ulu(T_v,c_v)$ and $A$ is a colimit of the latching diagram. We will focus first on the map $A\to B$, which we denote by $f_n^{T,c}$. Observe that $f_n^{T,c} = f_{n_1}^{T_1,c_1} \boxprod \dots \boxprod f_{n_v}^{T_v,c_v}$, where $n_i$ denotes the number of inputs to tree $T_i$. The inductive hypothesis tells us each $f_{n_i}$ is a cofibration in $\sF_{n_i}^{T_i,c_i}$. The realizability hypothesis and a simple cellular induction guarantees that $f_n^{T,c}$ is a cofibration in $\sF_n^{T,c}$, since each $f_{n_i}$ is a cellular extension using cells in $\sF_{n_i}^{T_i,c_i}$.

We now deal with the presence of $K(n)$. Every family model structure on $\Top^G$ is a topological model category by \cite[Proposition 3.11]{stephan}. Furthermore, if $K(n)$ is cofibrant in $\sF_n$ then it is cofibrant in $\Top$, since it is a cell complex built from cells of the form $(G\times \Sigma_n)/H \times i$ where $i:S^{d-1}\to D^d$. This observation is used in the proof of \cite[Proposition 3.11]{stephan}. Hence, $K(n) \otimes -$ preserves $\sF_n^{T,c}$-cofibrations.

The same proof works in the other cases from \cite[Lemma 3.1]{BM09}. If the root of $T$ is uncolored and unitary, the latching map is a union of $K(1) \times (A \to B)$ with $\ast \times (X\to Y)$. When the root of $T$ is colored, the latching map is a union of $L(n)\times (A\to B)$ with $K(n)\times (X\to Y)$. Although we have abused the notation $A\to B$ and $X\to Y$ to refer to several maps, all of these maps $A\to B$ and $X\to Y$ are induced by pushout products of $f_{n_i}$ as before, so the realizability hypothesis shows that they are cofibrations in the $\sF_n^{T,c}$ model structure. Using that the $\sF_n^{T,c}$ model structure is topological, and $K(n), L(n)$, and $\ast$ are all cofibrant spaces, all latching maps are $\sF_n^{T,c}$-cofibrations, so Lemma \ref{lemma:aut-family-cof} guarantees that $P(n)\to P[u](n)$ is an $\sF$-cofibration as required.
\end{proof}

This completes the proof of the Blumberg-Hill conjecture, and also provides us with a wealth of examples of $N_\infty$-operads.

\begin{example}
Taking only the trivial $\rho$ in each family $\sF_n$ recovers the result of \cite[Proposition~5.1]{BM03} on the operad $E_\infty^G$, when the base model category is $\Top^G$. Taking $G$ trivial recovers the result of \cite[Proposition 5.1]{BM03} on $E_\infty$-operads, when the base model category is $\Top$.
\end{example}

\begin{remark}
In this section, we restricted attention to families of graph subgroups (so that we could assemble a graph subgroup of $\sF_n$ from graph subgroups of $\sF_k, \sF_{n_1}, \dots, \sF_{n_k}$), but could have worked more generally. For any sequence $\sF = (\sF_n)$, we can combine families $\sF_a$ and $\sF_b$ to obtain families of subgroups of $G \times \Sigma_a \times \Sigma_b$ of the form $\{T \;|\; \pi_a(T) \in \sF_a, \pi_b(T) \in \sF_b\}$, where $\pi$ denotes the natural projection. 
Such subgroups have the property that $(X\times Y)^T \cong X^{\pi_a(T)}\times Y^{\pi_b(T)}$, and one could attempt to generalize Definition \ref{defn:realizable} so that families of subgroups built in this way satisfy the requisite containment in order for operad composition maps to exist (thereby characterizing the spaces of an $E_\infty^\sF$-operad as universal classifying spaces). While we have not pursued this line of inquiry here, there are some simple cases where such containments are trivially satisfied. If the families $\sF_n = 1\times 1$ for all $n$ are used, then the requisite compatibility between $P(n)$ for different $n$ is automatically satisfied. The same is true if $\sF_n$ is the family of all subgroups of $G\times \Sigma_n$, for every $n$. Lastly, if the families $\sF_n$ are all subgroups of the form $1\times K$ where $K < \Sigma_n$, then this approach recovers results of Rezk in the simplial setting, e.g., \cite[Proposition 3.5.1]{rezk-phd}.
\end{remark}

\section{Admissibility of family operads} \label{sec:admissibility}

In this section we provide model structures on categories of algebras over $E_\infty^\sF$-operads, in $G$-spaces and $G$-spectra. The $G$-spectra results generalize results stated in \cite[Appendix A]{blumberg-hill}. As the proofs for $G$-spectra use results from \cite[Appendix B]{kervaire}, our results about $G$-spectra will require $G$ to be finite. We begin by recalling some terminology.

Suppose $\M$ is a $\cat{V}$-model category and $\cat{V}$ is a monoidal model category. For every operad $P$ in $\cat{V}$, there is a category of $P$-algebras in $\M$, and the usual free-forgetful adjunction $(P,U)$.

\begin{definition}
We define a map of $P$-algebras $f$ to be a weak equivalence (resp. fibration) if $U(f)$ is a  weak equivalence (resp. fibration) in $\M$. We say there is a \textit{transferred} model structure on $P$-algebras if these two classes of maps (with cofibrations defined via lifting) satisfy the model category axioms. In this case, we call $P$ \textit{admissible}.
\end{definition}

\begin{proposition} \label{prop:admissibility-G-top}
Let $G$ be a compact Lie group. For any $G$-operad $P$, the category of $P$-algebras in $\Top^G$ has a transferred model structure.
\end{proposition}

\begin{proof}
The category $\Top^G$ is Cartesian and has the same interval object as $\Top$, so the conclusion follows from \cite[Proposition 4.1]{BM03}. While it is not true that all objects are small, it is true that the domains of the generating (trivial) cofibrations of $P$-algebras are small relative to the class of maps $P(I)$-cell (resp. $P(J)$-cell) where $P(-)$ is the free $P$-algebra functor. This is because such maps are inclusions, by an argument similar to the proof of Theorem \ref{thm:family-oper}.
\end{proof}

\begin{remark}
The proof also works for the $\sH$-model structure on $\Top^G$ for any family $\sH$. An alternative proof of Proposition \ref{prop:admissibility-G-top} can be obtained by using \cite[Propostion 2.3]{salvatore-wahl} and the results of \cite{BM03} applied to $\Top$.
\end{remark}

We now consider $G$-spectra, and we determine which operads $E_\infty^\sF$ have a transferred model structure on their algebras. In order to use the results from \cite{kervaire}, we must now restrict to finite $G$ for the remainder of the section. First, \cite[Theorem A.1]{blumberg-hill} states that there is a transferred model structure on algebras over any $N_\infty$-operad in $\Sp^G$ whose spaces have the homotopy type of a $G\times \Sigma_n$-CW complex. We now extend this result to certain operads $E_\infty^\sF$, and en route we record a number of observations to allow the reader to work with either the positive model structure (\cite[Theorem III.5.3]{mandell-may-equivariant}) or the positive complete model structure on $\Sp^G$. The positive complete model structure was introduced in \cite[Section B.4.1]{kervaire}, to fix a gap in \cite{mandell-may-equivariant}, as pointed out in \cite[Remark B.119]{kervaire}. Note that admissibility results for $N_\infty$-operads are false in general for non-positive model structures on $\Sp^G$.

We begin with commutative monoids, to sketch the main ideas of  \cite[Appendix~B]{kervaire}, and because we will need transferred model structures on commutative monoids in Section \ref{sec:rectification}. In \cite{white-commutative-monoids}, the second author introduced the {\emph commutative monoid axiom}, and proved that it implies that there is a transferred model structure on commutative monoids. Later, in \cite{white-yau1}, the second author and Donald Yau generalized and improved this axiom, to prove admissibility results for general colored operads. With that work in mind, we now define:

\begin{definition}
A model category is said to satisfy the \textit{generalized commutative monoid axiom} if, for all trivial cofibrations $f$, maps of the form $f^{\boxprod n}/\Sigma_n$ are contained in a class that is closed under transfinite composition, pushout, and smashing with arbitrary objects, and which is contained in the weak equivalences.
\end{definition}

\cite[Lemma B.108]{kervaire} proves that the positive complete model structure on $\Sp^G$ satisfies this generalization of the commutative monoid axiom (the maps $f^{\boxprod n}/\Sigma_n$ are trivial $h$-cofibrations), and \cite[Proposition B.131]{kervaire} uses this to transfer a model structure to commutative monoids. The proof of \cite[Lemma B.108]{kervaire} requires what is called the \emph{rectification axiom} (for $\Sp^G$) in \cite{white-commutative-monoids}, i.e., for all cofibrant $X$, the map $(E_G \Sigma_n)_+ \wedge X^{\wedge n} \to X^{\wedge n}/\Sigma_n$ is a weak equivalence.

The proof of the rectification axiom involves proving that the functors $(E_G \Sigma_n)_+ \wedge (-)^{\times n}$ preserve trivial cofibrations between cofibrant objects. To prove this requires writing the $G\times \Sigma_n$-CW complex $E_G \Sigma_n$ as a homotopy colimit of transitive $G\times \Sigma_n$-sets which are $\Sigma_n$-free (\cite[Lemma B.114]{kervaire}). While the verification of the rectification axiom in \cite{kervaire} is set in the positive complete model structure, it is a consequence that the rectification axiom, and hence the generalized commutative monoid axiom, also hold for the positive stable model structure \cite[Theorem III.5.3]{mandell-may-equivariant}, as we now show (following the phrasing of \cite{kervaire}). This statement is also required in \cite[Equation A.5]{blumberg-hill} in order to transfer a model structure to algebras over an $N_\infty$-operad.

\begin{proposition} \label{prop:rect-axiom-for-positive}
For all positive cofibrant $G$-spectra $X$, for all indexing sets $I$, and for all $\Sigma$ acting on $I$, the map $(E_G \Sigma)_+ \wedge_{\Sigma} X^{\wedge I} \to X^{\wedge I}/\Sigma$ is a positive stable weak equivalence. In particular, the map $(E_G \Sigma_n)_+ \wedge_{\Sigma_n} X^{\wedge n} \to X^{\wedge n}/\Sigma_n$ is a weak equivalence.
\end{proposition}

\begin{proof}
Every generating cofibration in the positive model structure on $\Sp^G$ is a positive complete cofibration, since the consideration of all $H$-representations $V$ for every $H < G$ includes the consideration of all $G$-representations (and the positivity condition for $G$-representations matches the complete positivity condition for $H$-representations when $H=G$). Thus, any positive cofibrant $X$ must be positive complete cofibrant, because it's built from a subcollection of the generating cells corresponding only to those where $H=G$. Next, these two model categories have the same weak equivalences, namely the genuine stable equivalences. Thus, for any positive cofibrant $X$, the map $(E_G \Sigma)_+ \wedge_{\Sigma} X^{\wedge I} \to X^{\wedge I}/\Sigma$ is a positive complete stable weak equivalence \cite[Proposition B.116]{kervaire}, hence a positive stable weak equivalence.
\end{proof}

\begin{remark} \label{remark:rect-for-general-N-infty}
For any $N_\infty$-operad $P$, the spaces $P(n)$ can be decomposed into a homotopy colimit just as in the case for $E_G \Sigma_n$ (the $n^{th}$ space of a \textit{complete} $N_\infty$-operad). It follows that the analogues of \cite[Proposition B.112]{kervaire} and \cite[Lemma B.114]{kervaire} hold with $P(n)$ in place of $E_G \Sigma_n$, and hence that $P$-algebras have a transferred model structure with respect to either the positive or positive complete model structure on $\Sp^G$.
\end{remark}

We now prove the $\sH$-version of \cite[Theorem A.1]{blumberg-hill}, where $\sH$ is a family of subgroups of $G$. We begin with a definition, of a type of operad that arises naturally when studying the $\sH$-family model structures on spaces and spectra. Such model structures often arise in work related to the Baum--Connes and Farrell--Jones conjectures \cite{luck-survey-classifying-families}. We begin with a definition, of a class of operads that parameterizes partial multiplications as in Example \ref{ex:restrictions}, but with $\Sigma_n$-freeness like in Definition \ref{defn:N-infty}.

\begin{definition} \label{defn:H-N-infty}
Fix a family $\sH$ of subgroups of $G$. A $G$-operad $P$ is called an \textit{$\sH$-$N_\infty$-operad} if $P(n)$ is a universal classifying space for a family $\sN_n(P)$ of subgroups of $G\times \Sigma_n$ containing all $H \times 1$ for $H\in \sH$, and intersecting $1\times \Sigma_n$ trivially. Furthermore, call $P$ a \textit{complete $\sH$-$N_\infty$-operad} if the families $\sN_n(P)$ consist of graphs for all $H\in \sH$ and all $\rho: H\to \Sigma_n$.
\end{definition}

\begin{remark}
The proof of the Blumberg-Hill conjecture in Section \ref{sec:cof-rep-operads} also proves that $\sH$-$N_\infty$-operads exist for any $\sH$-realizable sequence, defined just as in Definition \ref{defn:realizable}, but only with reference to $H\in \sH$. We leave the details to the reader.
\end{remark}

\begin{proposition} \label{prop:family-admissible}
Let $\sH$ be a family of subgroups of $G$, and let $P$ be an $\sH$-$N_\infty$-operad whose spaces have the homotopy type of $G\times \Sigma_n$-CW complexes. Then $P$ is admissible for the positive complete (hence for the positive) $\sH$-model structure $\Sp^\sH$.
\end{proposition}

We begin the proof with a lemma, generalizing \cite[Lemma B.114, Proposition B.112, and Proposition B.108]{kervaire} to the setting of $\Sp^\sH$ and $\sH$-$N_\infty$-operads. Note that, for the third result, taking $\sH = \{e\}$ recovers \cite[Lemma 15.5]{MMSS}. Note that taking $\sH = \{H < G\}$ recovers \cite[Equation A.5]{blumberg-hill} (correcting an error in the statement of \cite[Lemma III.8.4]{mandell-may-equivariant}).

\begin{lemma} \label{lemma:rectification-property-family}
Let $\sH$ be a family of subgroups of $G$, and give $\Sp^\sH$ either the positive or positive complete model structure. Let $P$ be an $\sH$-$N_\infty$-operad having the homotopy type of $G\times \Sigma_n$ CW-complexes.
\begin{itemize}
\item[{\rm (i)}] The functors $P(n)_+ \wedge_{\Sigma_n} (-)^{\times n}$ preserve trivial cofibrations between cofibrant objects (in the $\Sp^\sH$-model structure).
\item[{\rm (ii)}] The functors $P(n)_+ \wedge_{\Sigma_n} (-)^{\times n}$ preserve weak equivalences between cofibrant objects (in the $\Sp^\sH$-model structure).
\item[{\rm (iii)}] If $P$ is a complete $\sH$-$N_\infty$-operad and $X$ is cofibrant in $\Sp^\sH$, then the natural map $P(n)_+ \wedge_{\Sigma_n} X^{\wedge n} \to X^{\wedge n}/\Sigma_n$ is a weak equivalence in $\Sp^\sH$.
\end{itemize}

\end{lemma}

\begin{proof}
The proof of \cite[Lemma B.114]{kervaire} goes through with $P(n)$ instead of $E_G \Sigma_+$, using the equivariant cellular filtration of $P(n)$, i.e., using that $P(n)$ is cofibrant in the $\sF$-model structure on $\Top^\sF$, where $\sF$ is the family underlying $P$. This proves the first assertion. The second assertion is simply Ken Brown's lemma. The third assertion is proven just as in the proof of \cite[Lemma B.108]{kervaire}, using the completeness assumption for the analogue of \cite[Proposition B.116]{kervaire} (which now only yields a weak equivalence in $\Sp^\sH$). 
\end{proof}

We now prove Proposition \ref{prop:family-admissible}.

\begin{proof}
The proof follows the approach of \cite[Appendix B]{kervaire} as summarized above, with a clever trick from \cite[Proposition A.1]{blumberg-hill}. For every generating trivial cofibration $j:A\to B$ in $\Sp^\sH$, and any map $A \to U(X)$ in $\Sp^G$, we analyze the pushout morphism $f:X\to Y$
\begin{equation} \label{diagram:pushout}
\xymatrix{
P(A) \ar[r] \ar[d] & P(B) \ar[d] \\
X \ar[r] & Y}
\end{equation}
in the category of $P$-algebras. By \cite[Proposition 5.13]{MMSS}, what must be shown is that maps in $P(I)$ and $P(J)$ satisfy the cofibration hypothesis and that transfinite compositions of maps of the form $X\to Y$ above are weak equivalences. The first part is easy, using the coproduct decomposition of $P(A)\to P(B)$ and using \cite[Proposition B.89]{kervaire}. For the second part, we follow \cite{blumberg-hill} and show that $P(A) \to P(B)$ is a trivial $h$-cofibration \cite[Definition A.60]{kervaire} in the category of $P$-algebras. It follows that transfinite compositions of pushouts of such maps are trivial $h$-cofibrations (\cite[Theorem 8.12]{MMSS}). We deduce that $P(A)\to P(B)$ is a weak equivalence from Lemma \ref{lemma:rectification-property-family}(i), using the coproduct decomposition of $P(-)$, and the fact that the domains of the generating trivial cofibrations in $\Sp^\sH$ are cofibrant. 
\end{proof}

As a consequence of the proof, we also have:

\begin{corollary}
Both the positive complete and the positive model structure on $\Sp^\sH$ satisfy the generalized commutative monoid axiom, and maps of the form $f^{\boxprod n} /\Sigma_n$ are contained in the trivial $h$-cofibrations (trivial with respect to the $\sH$-model structure).
\end{corollary}

\begin{remark} \label{remark:do-not-use-fresse}
In this section, we focused on transferred model structures. However, experience has shown that often one only expects a transferred \textit{semi}-model structure (defined in \cite{spitzweck-thesis} and treated in detail in \cite{fresse-book}), where lifting and factorization only hold for maps with cofibrant domains. Indeed, \cite[Example 2.8]{batanin-white} presents a situation where there is no model structure on a category of operad-algebras, but there is a semi-model structure. The reader may be tempted to rely on \cite[Theorem 12.3.A]{fresse-book} to obtain transferred semi-model structures on $E_\infty^\sF$-algebras, for any sequence $\sF$ encoding free $\Sigma_n$-actions. Such operads are all $\Sigma$-cofibrant when viewed in the Berger-Moerdijk model structure on operads \cite{BM03}, and any operad in $\Top^G$ (or even in $\Top$) can act in $\Sp^G$. The reader is warned: using those methods will not work to prove the existence of a transferred semi-model structure.

More precisely, the proof of \cite[Theorem 12.3.A]{fresse-book} for an operad $P$ requires that $P(n) \wedge_{\Sigma_n} (-)^{\wedge n}$ preserves weak equivalences between cofibrant objects. Non-equivariantly, this is proven using a $\Sigma_n$-equivariant cellular filtration, and using the $\Sigma_n$-freeness of $E\Sigma_n$. Equivariantly, we need $G\times \Sigma_n$-cells, and we need to work with $E_G \Sigma_n$. When applied to $P = E_\infty^\sF$, the proof of \cite[Theorem 12.3.A]{fresse-book} only treats the trivial $\rho: H\to \Sigma_n$. Thus, the types of algebras obtained have no multiplicative norms. The results of \cite{fresse-book} are correct, but only produce a transferred semi-model structure on $E_\infty^G$-algebras (or $E_\infty$-algebras if the $\Top$ enrichment is used).

\end{remark}

\begin{remark} \label{remark:levelwise-family-model-on-spectra}
In this section, we proved admissibility results for $N_\infty$-operads and $\sH$-$N_\infty$-operads (with respect to $\Sp^\sH$). It is natural to wonder if, for any sequence $\sF$, there is some model structure on $\Sp^G$ where $E_\infty^\sF$-algebras will have a transferred model structure, analogous to how we used the $\sH$-model structure to get admissibility for $\sH$-$N_\infty$-operads. Certainly one can define a levelwise model structure on sequences of $G$-spaces, relative to the families $\sF_n$, but more work would be required to promote this to a stable model structure on $\Sp^G$.

\end{remark}

\section{Rectification for $E_\infty^\sF$-operads} \label{sec:rectification}

In this section we consider when weak equivalences of $N_\infty$-operads induce Quillen equivalences on their categories of algebras. In order to use the results of the previous section, we assume $G$ is finite throughout this section. First, observe that for $\Top^G$, we do not expect rectification results in general, since (even when $G$ is trivial), $E_\infty$-algebras are not Quillen equivalent to commutative monoids. However, for $G$-spectra, we do have rectification between different $N_\infty$-operads \cite[Theorem A.3]{blumberg-hill}. We now give a similar result for $\sH$-$N_\infty$-operads in $\Sp^\sH$, generalizing \cite[Theorem A.3]{blumberg-hill}. 

\begin{proposition} \label{prop:H-N-infty-rect}
Let $\sH$ be a family of subgroups of $G$ and give $\Sp^\sH$ either the positive or positive complete model structure. Let $P$ and $P'$ be $\sH$-$N_\infty$-operads, with sequences $\sF$ and $\sF'$ respectively, with $P(1)$ and $P'(1)$ having nondegenerate $G$-fixed basepoints, and whose spaces have the homotopy type of $G\times \Sigma_n$-CW-complexes. Let $f:P\to P'$ be a weak equivalence in the $\sF'$-model structure on $G$-operads. Then $f$ induces a Quillen equivalence on the model structures of $P$ and $P'$ algebras of Proposition \ref{prop:family-admissible}.
\end{proposition}

\begin{proof}
The assumption on $f$ guarantees that $\sF_n \subset \sF_n'$ for all $n$, because for any $\Gamma \in \sF_n$, the map $f_n^{\Gamma}$ cannot have empty codomain. Indeed, the assumption on $f$ guarantees that $\sF = \sF'$, because for all $\Gamma' \in \sF_n'$, the map $f_n^{\Gamma'}$ cannot have empty domain if it is to be a weak equivalence in $\Top$.

As in \cite[Theorem A.3]{blumberg-hill}, we appeal to \cite[Proposition 5.14]{ABGHR}, which carries out a detailed cellular induction argument whose base case consists of a comparison of free algebras. For us, this comparison uses Lemma \ref{lemma:rectification-property-family}(i). The rest of the cellular induction is formal, as long as we keep in mind that we never take fixed points for $H \notin \sH$.
\end{proof}

\begin{remark}
As in Remark \ref{remark:do-not-use-fresse}, the reader is warned not to rely on non-equivariant rectification results in the equivariant context. For example, while \cite[Theorem 12.5.A]{fresse-book} does provide Quillen equivalences of semi-model categories of algebras over $\Sigma$-cofibrant operads, forgetting to this setting loses the information of the multiplicative norms. If \cite[Theorem 12.5.A]{fresse-book} is used, it only says that $E_\infty^G$-operads have Quillen equivalent categories of algebras (or $E_\infty$-operads if the $\Top$-enrichment is used), not that $N_\infty$-operads are all weakly equivalent.
\end{remark}

We turn now to strictification, i.e., Quillen equivalences between $E_\infty^\sF$-algebras and strict commutative monoids.  \cite[Theorem A.6]{blumberg-hill} proves a related result for complete $N_\infty$-operads, but does not obtain a Quillen equivalence. We now prove the analogue for $\sH$-$N_\infty$-operads, and along the way we improve \cite[Theorem A.6]{blumberg-hill} to obtain a Quillen equivalence (with respect to either the positive or positive complete model structure on $\Sp^G$).

\begin{theorem} \label{thm:family-rectification}
Let $P$ be a complete $\sH$-$N_\infty$-operad. Then rectification holds between $P$ and $\Com$ in the positive complete (or positive) model structure $\Sp^\sH$, i.e., the unique map $f: P\to Com$ to the terminal $G$-operad induces a Quillen equivalence $f_!: \Alg_{E_\infty^\sF}(\Sp^\sH) \leftrightarrows \Alg_{Com}(\Sp^\sH):f^*$.
In particular, given a $P$-algebra $X$ there is a strictly commutative algebra $\tilde{X}$ and a weak equivalence $X\simeq \tilde{X}$.
\end{theorem}

\begin{proof}
Following \cite[Theorem 1.6]{goerss-hopkins-moduli-spaces}, it is easy to show that $(f_!,f^*)$ is a Quillen pair using the definition of the weak equivalences and fibrations in both categories, since weak equivalences and fibrations are created by the forgetful functor to $\Sp^\sH$. To show it is a Quillen equivalence we must show that for all cofibrant $E_\infty^\sF$-algebra $X$, the unit of the adjunction $X\to f^*f_!X$ is a weak equivalence. Reduce to the case where $X = P(X_0)$ via a standard cellular induction. The map $X\to f^*f_!X$ is now exactly the map induced by $f$ which goes from $P(X_0) \simeq \bigvee (E_\sF \Sigma_n)_+ \wedge_{\Sigma_n} X_0^{\wedge n}$ to $\Sym(X_0) = \bigvee X_0^{\wedge n}/\Sigma_n$. Lemma \ref{lemma:rectification-property-family}(iii) now proves this map is a weak equivalence.
\end{proof}

\begin{remark}
Taking $\sH$ to be the family of all subgroups of $G$ improves \cite[Theorem~A.6]{blumberg-hill} to obtain a Quillen equivalence.
\end{remark}

\begin{remark}
An alternative approach to Theorem \ref{thm:family-rectification} is to appeal to \cite[Theorem~4.6]{white-commutative-monoids}, where it is proven that, if the rectification axiom holds (as is shown here in Proposition~\ref{prop:rect-axiom-for-positive} and Lemma \ref{lemma:rectification-property-family}(iii)), then $f$ induces a Quillen equivalence between commutative monoids and $QCom$-algebras, which in this case are precisely $P$-algebras for $P$ as in the statement of Theorem~\ref{thm:family-rectification}. \cite[Theorem 4.6]{white-commutative-monoids} has been generalized in \cite{white-yau3}, which also provides a mechanism for lifting Quillen equivalences among different models of $G$-spectra to Quillen equivalences of $P$-algebras.
\end{remark}

\section{Bousfield Localization and $E^\sF_\infty$-structure} \label{sec:localization}

In this section we give an application of the work from the previous
sections. We work with a compact Lie group $G$, but we remind the reader that results regarding the existence of transferred (semi-)model structures on commutative equivariant ring spectra and algebras over $N_\infty$-operads, and results regarding rectification, are only presently known for the setting of a finite group $G$. We begin with an example \cite[Theorem 2]{mcclure-e-infinity-tate}, that we learned from Mike Hill, demonstrating that certain localizations of $\Sp^G$ can take $E_\infty^\sF$-algebras to $E_\infty^G$-algebras, destroying all norm structure. This example is expounded in \cite[Example 5.7]{white-localization}, but we recall the main details.

\begin{example} \label{example:hill}
Let $\sP$ be the family of proper subgroups of $G$ and let $\tilde{E}\sP$ be the cofiber of the natural map from the classifying space $E\sP_+$ to $S^0$. This $\tilde{E}\sP$ is a localization of $S^0$ obtained by killing all maps from induced cells. If $G$ is finite, then this spectrum $E = \tilde{E}\sP$ does not admit multiplicative maps from the norms of its restrictions, and so cannot be a commutative monoid (nor an algebra over any $N_\infty$-operad with norms linking proper subgroups of $G$ with $G$), even though it is a localization of the commutative $G$-ring spectrum $S^0$. The proof that $E$ cannot admit multiplicative norms uses the fact that the restriction $\res_H(E)$ is contractible for every $H \in \sP$, hence there cannot be ring homomorphisms $N_H^G \res_H(E) \to E$ as would be required if $E$ had norms. These maps cannot exist, because $E$ is not contractible (because $E\sP_+$ is not homotopy equivalent to $S^0$, because $\sP$ doesn't contain $G$). So while any restriction to a proper subgroup views them to be homotopy equivalent, they are not homotopy equivalent in $\Sp^G$. This example is a localization that takes every $P$-algebra to an $E_\infty^G$-algebra, for any $N_\infty$-operad $P$. We denote this localization $L_{\sP}$.
\end{example}

Indeed, something similar occurs non-equivariantly. The Postnikov section map demonstrates that a localization can even destroy $E_\infty$ structure \cite[Section 6]{CGMV}. We prohibit this behavior by assuming our localizations commute with non-equivariant suspension (see also \cite{Gut12} for more results in this direction). In the language of \cite{CGMV} such localizations are called \textit{closed}. In the language of
\cite{white-localization} they are called \textit{monoidal localizations}. Recall that, for a set of morphisms $C$ in a model category $\M$, the left Bousfield localization $L_C(\M)$ is a new model structure on $\M$ with more weak equivalences (called $C$-local equivalences), the same cofibrations, and satisfying the universal property that any left Quillen functor $F:\M\to \cat N$, taking the morphisms in $C$ to weak equivalences, factors through the identity $id:\M\to L_C(\M)$. Hirschhorn \cite{hirschhorn} proves the existence of the model structures $L_C(\M)$ if $\M$ is left proper and cellular. These conditions on $\Sp^G$ are verified in Appendix \ref{appendix}. Monoidal localizations are characterized in \cite[Theorem 4.5]{white-localization}. We state the version of that theorem for $\Sp^\sH$.

\begin{theorem} \label{thm:monoidal-loc-G-spectra}
Let $G$ be a compact Lie group. Suppose $C$ is a set of cofibrations between cofibrant objects in $\Sp^\sH$. Then the model category $L_C(\Sp^\sH)$ is a monoidal model category if and only if the class of $C$-local equivalences is closed under suspension, and for all $H\in \sH$ maps of the form $C\wedge F_W((G/H)_+ \wedge S^{n-1}_+)$ are $C$-local equivalences for all $W$ in the universe on which $\Sp^\sH$ is indexed.
\end{theorem}

For such localizations, the additional equivariant norm structure can be destroyed, but the baseline $E_\infty^\sH$ structure is always preserved, where $E_\infty^\sH$ denotes the operad of Definition~\ref{defn:E-infty-F-operad} with respect to the families $\sF_n = \{H \times 1\;|\; H\in \sH\}$. As a special case (taking $\sH$ to be all subgroups of $G$), we prove that any monoidal localization of a $P$-algebra, for an $N_\infty$-operad $P$, has an $E_\infty^G$-structure.

\begin{defn} \label{defn:preservation-white}
$L_C$ is said to \textit{preserve $P$-algebras} if

\begin{enumerate}
\item When $E$ is a $P$-algebra there is some $P$-algebra $\tilde{E}$ which is weakly equivalent in $\M$ to $L_C(E)$.

\item In addition, when $E$ is a cofibrant $P$-algebra, then there is a choice of $\tilde{E}$ in $\Alg_P(\M)$ with $U(\tilde{E})$ local in $\M$, there is a $P$-algebra homomorphism $r_E:E\to \tilde{E}$ that lifts the localization map $l_E:E\to L_C(E)$ up to homotopy, and there is a weak equivalence $\beta_E:L_C(UE)\to U\tilde{E}$ such that $\beta_E \circ l_{UE} \cong Ur_E$ in $\ho(\M)$.
\end{enumerate}
\end{defn}

To understand when localization preserves structure, we make use of \cite[Corollary 3.4 and Theorem 5.1]{white-localization}, restated to $G$-spectra, and weakened to avoid the language of semi-model structures. In the following, $G$ is a compact Lie group, though to verify the existence of transferred model structures in practice often requires $G$ to be finite.

\begin{theorem} \label{thm:localization-preserves}
Suppose $C$ is a class of maps in $\Sp^\sH$ for which the Bousfield localization $L_C(\M)$ exists and is a monoidal model category. Let $P$ be an operad such that the categories of $P$-algebras in $\Sp^\sH$ and in $L_C(\Sp^\sH)$ inherit model structures from $\Sp^\sH$ and $L_C(\Sp^\sH)$. Then $L_C$ preserves $P$-algebras. Furthermore, if $Q$ is a $\Sigma$-cofibrant $G$-operad, with respect to $\Oper(\Top^\sH)$, then any monoidal left Bousfield localization $L_C$ preserves $Q$-algebras.
\end{theorem}

\begin{corollary} \label{cor:loc-pres-Sigma-cof}
For any $\sH$-$N_\infty$-operad $P$, any $P$-algebra $X$, and any monoidal left Bousfield localization $L_C$, the object $L_C(X)$ is weakly equivalent to an $E_\infty^\sH$-algebra.
\end{corollary}

\begin{proof}
There is a forgetful functor from $\Alg_P(\Sp^\sH)$ to $\Alg_{E_\infty^\sH}(\Sp^{\sH})$ that forgets all nontrivial $\rho$ (i.e., forgets all norms). As discussed in Remark \ref{remark:do-not-use-fresse}, the operad $E_\infty^\sH$ is $\Sigma$-cofibrant relative to the Berger--Moerdijk model structure on $\Oper(\Top^\sH)$ of \cite{BM03}. By Theorem~\ref{thm:localization-preserves}, with $Q = E_\infty^\sH$, $L_C(X)$ is weakly equivalent to an $E_\infty^\sH$-algebra.
\end{proof}

Unlike the non-equivariant setting, equivariant localizations can be ``partially monoidal'', i.e., maps in $C\wedge F_W((G/H)_+ \wedge S^{n-1}_+)$ can be $C$-local equivalences for some $W$ and some $H$, even if they are not $C$-local equivalences for all $W,H$.

\begin{remark} \label{remark:lattice-of-operads}
For any sequence of families of subgroups of $G\times \Sigma_n$, we have $E_\infty^\sF$-operads via Definition \ref{defn:E-infty-F-operad}. For each $n$, there is a poset of families of subgroups of $G\times \Sigma_n$, ordered by inclusion. For any subfamily $\sF_n \subset \sF_n'$, an $E_\infty^{\sF'}$-algebra may be viewed as an $E_\infty^\sF$-algebra via neglect of structure. Hence, we have a poset of $E_\infty^\sF$-operads, interpolating between the minimal amount of structure (discussed in Example \ref{ex:recovering-BM-model}) and the $E_\infty^\sF$-operad corresponding to the sequence $\sF = (\sF_n)$ where $\sF_n$ is all subgroups of $G\times \Sigma_n$. Non-monoidal localizations can take algebras over an $E_\infty^{\sF'}$-operad to a lower $E_\infty^\sF$-operad (for $\sF \subset \sF'$).

Within this poset of $E_\infty^\sF$-operads is contained a poset of $N_\infty$-operads (and, more generally, $\sH$-$N_\infty$-operads). This poset interpolates between $E_\infty^G$ (which has no multiplicative norm maps) and complete $N_\infty$-operads (which has all possible norms). Non-monoidal localization can also move algebras within this poset. Worse, even monoidal localizations can move algebras within this poset, as Example \ref{example:hill} shows. This problem is remedied in Theorem \ref{thm:preserving-N-infty} below. The observation that Example \ref{example:hill} is smashing, hence monoidal, is due to Justin Noel.
\end{remark}

\begin{example} \label{example:lattice-postnikov}
We have seen that monoidal localizations of $N_\infty$-operad algebras always land at least in $E_\infty^G$-algebras. Generalizing the Postnikov section example \cite[Section 6]{CGMV}, we may define localizations $L$ that are monoidal with respect to any family $\sH$ of subgroups of $G$, but not monoidal with respect to all subgroups of $G$. These localizations drop structure to less than $E_\infty^G$. The following diagram represents the movement within the poset discussed in the previous remark, where all families take the form $\{H\times 1\;|\; H\in \sH\}$. The vertical bars represent the forgetful functors, and the families get smaller in lower levels of the diagram (i.e., $\sH \subset \sH'$). The arrows represent localization functors. For simplicity, we have compressed the poset to a linear order, though of course the actual posets need not be linear:

\begin{align*}
\xymatrix{E_\infty^G \ar@{-}[d] \ar[dddr] \ar[ddddr] \ar[dr] \ar[r]^L & E_\infty^G \ar@{-}[d] \\
E_\infty^{\sH'} \ar@{-}[d] \ar[r] \ar[dddr]& E_\infty^{\sH'} \ar@{-}[d]\\
\dots \ar@{-}[d]& \dots \ar@{-}[d] \\
E_\infty^{\sH} \ar@{-}[d] \ar[r] \ar[dr] & E_\infty^{\sH} \ar@{-}[d]\\
E_\infty \ar[r] & E_\infty.}
\end{align*}

Here $E_\infty$ means $E_\infty^\sT$ for the trivial family $\sT$. Any localization that is monoidal with respect to $\sH$ will land in $E_\infty^\sH$-algebras.
\end{example}

We have provided a theorem guaranteeing preservation of $E_\infty^G$-structure, and have given an example characterizing when such structure can be reduced to $E_\infty^\sH$-structure. We finish the paper with results guaranteeing preservation, even of $N_\infty$-structure.

In \cite[Corollary 6.7]{white-localization} (together with \cite[Theorem 5.6]{batanin-white}), localizations that preserve commutative monoid structure are characterized. We recall the version for $G$-spectra \cite[Theorem 7.9]{white-localization} as Theorem \ref{thm:preserving-N-infty} below. An alternative approach to preservation of commutative monoid structure (or preservation of algebras over linear isometries $G$-operads) may be found in \cite[Section 6]{hill-hopkins2}. Let $\Sym(-)$ denote the free commutative monoid functor. As all of our results from now on require the existence of transferred (semi-)model structures on categories of algebras, we assume $G$ is finite for the rest of the section.

\begin{theorem} \label{thm:preserving-N-infty}
Consider the positive complete (or positive) stable model structure on $G$-spectra. Suppose $L_C$ is a monoidal left Bousfield localization. Then the following are equivalent:
\begin{enumerate}
\item $L_C$ preserves commutative equivariant ring spectra,
\item $\Sym^n(-)$ preserves local equivalences between cofibrant objects, for all $n$,
\item $\Sym^n(-)$ takes maps in $C$ to local equivalences, and 
\item $\Sym^n(-)$ preserves $C$-acyclicity for all $n$.
\end{enumerate}
\end{theorem}

The following corollary gives easy to check conditions under which a left Bousfield localization will preserve algebras over $N_\infty$-operads that satisfy rectification with respect to $\Com$. The assumption on $P$ is only so that we can use \cite[Theorem A.6]{blumberg-hill} in order to rectify. 

\begin{corollary} \label{cor:loc-preserves-N-infty}
Let $P$ be a complete $N_\infty$-operad whose spaces have the homotopy type of $G\times \Sigma_n$ CW-complexes, and with a non-degenerate $G$-fixed basepoint. Let $C$ be a set of maps of $G$-spectra (with the positive or positive complete model structure). Then a monoidal localization $L_C$ preserves $P$-algebras if and only if $\Sym^n(C)$ is contained in the $C$-local equivalences for all $n$. Furthermore, such localizations preserve $P$-algebra structure for general $N_\infty$-operads, and preserve $P$-algebras for $\sH$-$N_\infty$-operads $P$, with respect to the $\sH$-model structure $\Sp^\sH$ on $\Sp^G$.
\end{corollary}

\begin{proof}
Observe that the rectification axiom of Section \ref{sec:admissibility} is unchanged by left Bousfield localization. If it holds in $\M$ then it holds in $L_C(\M)$, since cofibrant objects are the same, and every weak equivalence in $\M$ is a weak equivalence in $L_C(\M)$. Thus, localization preserves commutative monoids if and only if localization preserves algebras over a complete $N_\infty$-operad. The first statement of the corollary now follows from Theorem \ref{thm:preserving-N-infty}. 

The ``furthermore'' part follows from Corollary \ref{cor:best-preservation} below. The point is that the condition on $\Sym$ in Corollary \ref{cor:loc-preserves-N-infty} implies that $\Sym$ preserves $C$-local equivalences (by \cite[Corollary 6.7]{white-localization}). This preservation occurs if and only if all norm functors $N_H^G \res_H(-)$ preserve $C$-local equivalences (by Proposition \ref{prop:rect-axiom-for-positive}). In particular, all norms encoded by $P$ will preserve $C$-local equivalences, so Corollary \ref{cor:best-preservation} implies $L_C$ preserves $P$-algebras.

\end{proof}

In order to get preservation results for non-complete $N_\infty$-operads, we shift to phrasing our condition on $C$ in terms of the norm functors, since phrasing the condition in terms of $\Sym$ is the same as a requirement about {\em all} norm functors. The following is phrased in terms of $\sH$-$N_\infty$-operads, but already the case for $\sH = \{H < G\}$ is new (though a version where $P$ is a linear isometries $N_\infty$-operad has appeared in \cite[Section 6]{hill-hopkins2}). We use the notation $N^T(-)$ \cite[Definition 6.1]{blumberg-hill} for the norm parameterized by a $H$-set $T$. These norms can be written in terms of the norms $N_H^G \res_H(-)$ \cite[Proposition 6.2]{blumberg-hill}, but the notation becomes more cumbersome.

\begin{corollary} \label{cor:best-preservation}
Let $P$ be an $\sH$-$N_\infty$-operad whose spaces have the homotopy type of $G\times \Sigma_n$ CW-complexes. Let $L_C$ be a monoidal left Bousfield localization. Then $L_C$ preserves $P$-algebras in $\Sp^\sH$ if and only if, for all $H\in \sH$, the functors $G_+ \wedge_H N^T(-)$ preserve $C$-local equivalences between cofibrant objects, for all norms $N^T$ parameterized by $P$ (one for each homomorphism $H\to \Sigma_n$), i.e. for all admissible $H$-sets $T$. 
\end{corollary}

\begin{proof}
Assume that the functors $N^T(-)$ preserve $C$-local equivalences for all norms parameterized by $P$. To prove $L_C$ preserves $P$-algebras, we will use Theorem \ref{thm:localization-preserves}. First, $\Alg_P(\Sp^\sH)$ has a transferred model structure by Proposition \ref{prop:family-admissible}. Next, we can use the same proof in $L_C(\Sp^\sH)$ to  obtain a transferred model structure on $\Alg_P(L_C(\Sp^\sH))$. The free $P$-algebra maps $P(A)\to P(B)$ are still $h$-cofibrations. To prove they are $C$-local equivalences, we use our assumption on $C$, instead of Lemma \ref{lemma:rectification-property-family}. The norm functor is defined (\cite[Definition 6.1]{blumberg-hill}, \cite[Section 2.2.3]{kervaire}) so that $G_+ \wedge_H N^T (-) \cong ((G \times \Sigma_n)/\Gamma)_+ \wedge_{\Sigma_n} (-)^{\wedge n}$, where $\Gamma$ is the graph subgroup parameterizing the norm. Thus, our assumption, together with a standard cellular induction (writing $P(n)$ in terms of the cells $(G \times \Sigma_n)/\Gamma$), implies that $P(n)_+ \wedge_{\Sigma_n} (-)^{\wedge n}$ preserves $C$-local equivalences for all $n$. The proof of Proposition \ref{prop:family-admissible} now goes through in $L_C(\Sp^\sH)$, hence $L_C$ preserves $P$-algebras by Theorem \ref{thm:localization-preserves}.

For the converse, assume $L_C$ preserves $P$-algebras. Then \cite[Theorem 5.6]{batanin-white} proves the free $P$-algebra functor preserves $C$-local equivalences between cofibrant objects. The identification $G_+ \wedge_H N_H^G \res_H (-) \cong ((G \times \Sigma_n)/\Gamma)_+ \wedge_{\Sigma_n} (-)^{\wedge n}$ proves that every norm functor parameterized by $P$ preserves $C$-local equivalences between cofibrant objects.
\end{proof} 

Just as in Theorem \ref{thm:preserving-N-infty}, for monoidal localizations $L_C$, the functors $P(-)$ preserve $C$-equivalences if and only if they preserve $C$-acyclicity.

\begin{remark}
If the reader wishes to weaken the assumption in Corollary \ref{cor:best-preservation}, so that one only needs to check that the norm functors take $C$ into the class of $C$-local equivalences (rather than requiring the norm functors to preserve the class of $C$-local equivalences), then the model of \cite[Corollary 6.7]{white-localization} can be followed. For brevity's sake, we have not pursued that approach here.
\end{remark}

We finish with an example that generalizes Hill's example, and shows that it is possible to destroy some, but not all, equivariant norms. From this point of view, Hill's example is maximally bad, because of its use of the family $\sP$ of proper subgroups of $G$.

\begin{example} \label{ex:generalized-hill}
Consider the $\sH$-generalization of Hill's example, where we replace $\sP$ everywhere by a general family $\sH$ of subgroups of $G$. Denote the resulting localization $L_\sH$ and define $E$ in the analogous way to Example \ref{example:hill} (i.e., $E = L_\sH(S^0)$). Just as in Example~\ref{example:hill}, $L_\sH$ is a monoidal Bousfield localization (relative to the $\sH$-model structure on $\Sp^G$), hence cannot reduce structure to below $E_\infty^\sH$. However, even more is true. Now some, but not all, of the spectra $\res_H(E)$ will be contractible, so $E$ can admit some, but not all, norms.
Indeed, given any containment $\sG \supset \sH$ of families of subgroups of $G$, such a localization can be arranged to reduce $\sG$-$N_\infty$ algebra structure to $\sH$-$N_\infty$ algebra structure.

\end{example}

Between Corollary \ref{cor:loc-pres-Sigma-cof}, Corollary \ref{cor:best-preservation}, Example \ref{example:lattice-postnikov}, and Example \ref{ex:generalized-hill} we now have results characterizing when localization preserves $E_\infty^\sF$-algebra structure (both multiplicative norm structure and classical $E_\infty$ structure), and we have examples allowing any move between algebraic structures in the poset of $\sH$-$N_\infty$-operads.

\appendix

\section{} \label{appendix}

In this Appendix, we prove that the positive complete stable model structure on $G$-spectra is left proper and cellular, conditions sufficient for left Bousfield localizations to exist. Throughout, $G$ is a compact Lie group. Where results from \cite{kervaire} are used (which assumes $G$ is a finite group), we have checked that these results remain true in the setting of a compact Lie group $G$. The proof technique of Proposition \ref{prop:G-spectra-cellular} demonstrates that the non-positive and non-complete variants, as well as the variants where the model structure is defined relative to a family of subgroups $\sH$ of $G$, are also left proper and cellular. 
The proof here is based on \cite[Theorem A.9]{hovey-spectra}, and uses the description of $G$-spectra used in this paper, and discussed in \cite{hovey-white} (rather than the description used in \cite{mandell-may-equivariant}).

\begin{proposition}\label{prop:G-spectra-cellular}
The positive complete stable model structure $\Sp^G$ on $G$-spectra is left proper and cellular.
\end{proposition}

\begin{proof}
That it is left proper can be deduced from the combination of three results in \cite{kervaire}. First, in \cite[Remark B.64]{kervaire} it is proven that cofibrations are $h$-cofibrations, i.e., have the Homotopy Extension Property. In \cite[Corollary B.21]{kervaire} it is proven that $h$-cofibrations are flat maps (i.e., maps $f$ such that cobase change along $f$ preserves weak equivalences \cite[Definition B.9]{kervaire}). Finally, in \cite[Remark B.10]{kervaire}, it is observed that a model category is left proper if and only if every cofibration is a flat map. Putting these together we see that $\Sp^G$ is left proper.

We turn now to cellularity. We must prove that $\Sp^G$ satisfies the following three properties \cite[Chapter 12]{hirschhorn}, relative to the 
generating cofibrations $I = \{G_+ \wedge_H S^{-V} \wedge S^{n-1}_+ \to G_+ \wedge_H S^{-V} \wedge D^n_+ \}$ of the positive complete model structure (\cite[Definition B.4.1]{kervaire}) and the generating acyclic cofibrations $J$ defined analogously.
The three properties are as follows:
\begin{enumerate}
\item The domains and codomains of $I$ are compact relative to $I$.
\item The domains of $J$ are small relative to the cofibrations.
\item Cofibrations are effective monomorphisms.
\end{enumerate}

The easiest to verify is (2). We have just seen that the cofibrations are $h$-cofibrations, and \cite[Lemma A.70]{kervaire} shows that $h$-cofibrations are objectwise closed inclusions. A $\lambda$-sequence of $h$-cofibrations is again an $h$-cofibration  \cite[Proposition A.62]{kervaire}, and all spaces are small relative to inclusions \cite[Lemma 2.4.1]{hovey-book}. Let $X$ be a domain of a map in $J$, let $Y_\alpha$ be a $\lambda$-sequence of cofibrations in $\Sp^G$ for some regular cardinal $\lambda$, and let $f:X\to \colim Y_\alpha$. Each space $X_n$ is small relative to the closed inclusions making up the colimit, so $f_n$ factors through some map $g_n$ to an earlier stage $Y_{\beta_n}$. Take $\beta$ to be the supremum of the $\beta_n$ and the $g_n$ will assemble to a map $g:X\to Y_\beta$, verifying smallness.

Next we turn to (3). A map $f:X\to Y$ is an {\em effective monomorphism} if $f$ is the equalizer of the two obvious maps $Y \rightrightarrows Y\coprod_X Y$. In $\Top$ and $\Top^G$, condition (3) is equivalent to the statement that $X$ is the intersection of the two copies of $Y$ in $Y\coprod_X Y$. In particular, closed inclusions in (compactly generated weak Hausdorff) spaces are precisely the effective monomorphisms. Since limits in $\Sp^G$ are taken levelwise, it is sufficient to check condition (3) on each level $f_n$. Here again \cite[Lemma A.70]{kervaire} shows that $h$-cofibrations in $\Sp^G$ (hence cofibrations) are objectwise $h$-cofibrations in $\Top^G$, hence effective monomorphisms.

Finally, we turn to (1). The domains and codomains of maps in $I$ have the form $G_+ \wedge_H S^{-V} \wedge K_+$ where $K$ is either $S^{n-1}$ or $D^n$. Observe that $S^{-V} \wedge -$ is the left adjoint $F_V$ to a functor $Ev_V:\Sp^G \to \Top^H$, where $V$ is an $H$-representation. The functor $F_V$ takes an $H$-space $A$ to the $G$-space $A \wedge_H G_+ \cong A\wedge (G/H)_+$ and then to the spectrum $S^{-V} \wedge A\wedge (G/H)_+$. In the proof to follow, let $A=K_+$ denote a domain or codomain of a map in $I$, the set of generating cofibrations in $\Top^H$. Because everything is being converted into a $G$-space, it will not matter which $H$ our object begins with. Let $F_V(A)$ denote a domain or codomain of a generating cofibration in $\Sp^G$.

The notion of compactness in (1) is different from the one in \cite{hovey-book}. Here it means that there is some cardinal $\kappa$ such that for every relative $I$-cell complex $f:X\to Y$ and for every presentation of $f$ by a chosen collection of cells then every map $F_V(A)\to Y$ factors through a subcomplex of size at most $\kappa$. A presentation of $f$ is a realization of $f$ as the colimit of a $\lambda$-sequence of maps which are pushouts of coproducts of cells. A subcomplex of the given presentation of $f$ is a $\lambda$-sequence formed by pushouts of coproducts of a subset of cells. The size is the cardinality of the set of cells.

We will prove (1) by following the proof of \cite[Proposition A.8]{hovey-spectra}. This proof has three key ingredients. First, for all $H$ the standard fixed-point model structure on $\Top^H$ is cellular. This is because cells have the form $(H/K)_+ \wedge S^{n-1}_+ \to (H/K)_+ \wedge D^n_+$ and so the same proof which works for spaces applies just as well to this model category. Next, by adjunction every map $F_V(A)\to Y$ is equivalent to a map $A\to Ev_V(Y)$ and for every presentation $f:X\to Y$, we write $f_V:X_V\to Y_V$ as a retract of a presentation $X_V=Z^0 \to Z^1\to \dots \to Z^\lambda = Z$ of relative $I$-cell complexes \cite[Lemma 10.5.25]{hirschhorn}, where again $I$ is the set of generating cofibrations for $\Top^H$ and $V$ is an $H$-representation, in which every cell appearing in the presentation of $X_V\to Z$ has an associated $F_V(I)$-cell appearing in the presentation of $f$. This step allows us to reduce the verification of spectra to one in spaces, and all it requires is that every generating cofibration of spectra is an objectwise cofibration in spaces (this is clear, since every such map is a cofibrant object smashed with a cofibration of spaces).
Lastly, a transfinite induction following precisely the same steps as in \cite[Proposition A.8]{hovey-spectra} verifies that for every presentation of $f$ by a $\lambda$-sequence, if $\kappa$ is a cardinal for which $\Top^H$ satisfies cellularity for all $H$ then every map $F_V(A)\to Y$ factors through a subcomplex with at most $\kappa$ many $F_V(I)$-cells.

This transfinite induction makes use of the fact that we have already verified (3), so a subcomplex is uniquely determined by its set of cells \cite[Proposition 12.2.1]{hirschhorn}. It also uses the observation that it suffices to work with a $\lambda$-sequence formed by transfinite compositions of pushouts of maps in $I$ rather than coproducts of such maps. Lastly, the main points of the transfinite induction are that every $F_V(I)$-cell is contained in a subcomplex of at most $\kappa$-many $F_V(I)$-cells by induction (since each cell is glued on at some stage $\alpha < \lambda$), the $\lambda$-sequence in $\Top^H$ has the property that $A\to Y_n \to Z$ factors through a subcomplex with at most $\kappa$-many $I$-cells, and the subcomplex of $Y$ formed by the $F_V(I)$-cells corresponding to the $I$-cells required for the factorization in spaces still uses fewer than $\kappa^2 = \kappa$ many $F_V(I)$-cells. The interested reader is referred to Hovey's original proof in \cite[Proposition A.8]{hovey-spectra} for more details.
\end{proof}

\begin{remark}
The same proof demonstrates that the positive non-complete model structure and the stable model structure of \cite[Theorem III.4.2 and Theorem III.5.3]{mandell-may-equivariant} are cellular (using the equivalence of this model structure with the $\cat{U}$-stable model structure of \cite[Section 5]{hovey-white}), as well as the family model structures $\Sp^\sH$ of Section \ref{sec:prelim}. 
\end{remark}


\begin{thebibliography}{10}

\bibitem[ABGHR14]{ABGHR}
M.~Ando, A.~J. Blumberg, D.~Gepner, M.~J. Hopkins, and C.~Rezk.
\newblock Units of ring spectra, orientations, and {T}hom spectra via rigid infinite loop space theory.
\newblock {\em Journal of Topology}, 7:1077--1117, 2014.

\bibitem[BW16]{batanin-white}
Michael Batanin and David White.
\newblock Left {B}ousfield localization and {E}ilenberg-{M}oore categories,
  preprint available electronically from \texttt{arXiv:1606.01537}.
\newblock 2016.

\bibitem[BM03]{BM03}
Clemens Berger and Ieke Moerdijk.
\newblock Axiomatic homotopy theory for operads.
\newblock {\em Comment. Math. Helv.}, 78(4):805--831, 2003.

\bibitem[BM09]{BM09}
Clemens Berger and Ieke Moerdijk.
\newblock On the derived category of an algebra over an operad.
\newblock {\em Georgian Math. J.}, 16(1):13--28, 2009.

\bibitem[BH15]{blumberg-hill}
Andrew~J. Blumberg and Michael~A. Hill.
\newblock Operadic multiplications in equivariant spectra, norms, and transfers.
\newblock {\em Advances in Mathematics}, 285:658--708, 2015.


\bibitem[Boh12]{bohmann-overview}
Anna-Marie Bohmann.
\newblock Basic notions of equivariant stable homotopy theory, preprint available electronically from
  http://math.northwestern.edu/~bohmann/basicequivnotions.pdf.
\newblock 2012.

\bibitem[BP17]{bonventre-pereira} Peter Bonventre and Luis A. Pereira. 
Genuine equivariant operads, in preparation.

\bibitem[CGMV10]{CGMV}
Carles Casacuberta, Javier~J. Gutierrez, Ieke Moerdijk, and Rainer~M. Vogt.
\newblock Localization of algebras over coloured operads.
\newblock {\em Proc. Lond. Math. Soc. (3)}, 101(1):105--136, 2010.

\bibitem[CE14]{cort_ellis}
Guillermo Corti{\~n}as and Eugenia Ellis.
\newblock Isomorphism conjectures with proper coefficients.
\newblock {\em J. Pure Appl. Algebra}, 218(7):1224--1263, 2014.

\bibitem[Fau08]{fausk-equivariant-pro-spectra}
Halvard Fausk.
\newblock Equivariant homotopy theory for pro-spectra.
\newblock {\em Geom. Topol.}, 12(1):103--176, 2008.

\bibitem[Fre09]{fresse-book}
Benoit Fresse.
\newblock {\em Modules over operads and functors}, volume 1967 of {\em Lecture
  Notes in Mathematics}.
\newblock Springer-Verlag, Berlin, 2009.

\bibitem[GH04]{goerss-hopkins-moduli-spaces}
Paul~G. Goerss and Michael~J. Hopkins.
\newblock Moduli spaces of commutative ring spectra.
\newblock In {\em Structured ring spectra}, volume 315 of {\em London Math.  Soc. Lecture Note Ser.}, pages 151--200. Cambridge Univ. Press, Cambridge,
  2004.
  
  \bibitem[Gut12]{Gut12} Javier J. Guti\'errez.
\newblock Transfer of algebras over operads along Quillen adjunctions.
\newblock {\em J. Lond. Math. Soc.} (2) 86, no. 2, 607--625, 2012.


\bibitem[HH11]{oberwolfach}
Michael~A. Hill and Michael~J. Hopkins.
\newblock Localizations of equivariant commutative rings.
\newblock In {\em Mathematisches Forschungsinstitut Oberwolfach}, volume Report No. 46/2011, pages 2640--2643. 2011.


\bibitem[HH16]{hill-hopkins2} 
Michael~A. Hill and Michael~J. Hopkins.
Equivariant symmetric monoidal structures. Available as arxiv:1610.03114.

\bibitem[HHR16]{kervaire}
Michael~A. Hill, Michael~J. Hopkins, and Douglas~C. Ravenel.
\newblock On the nonexistence of elements of {K}ervaire invariant one.
\newblock {\em Annals of Mathematics}, 184:1--262, 2016.

\bibitem[Hir03]{hirschhorn}
Philip~S. Hirschhorn.
\newblock {\em Model categories and their localizations}, volume~99 of {\em
  Mathematical Surveys and Monographs}.
\newblock American Mathematical Society, Providence, RI, 2003.

\bibitem[Hov99]{hovey-book}
Mark Hovey.
\newblock {\em Model categories}, volume~63 of {\em Mathematical Surveys and
  Monographs}.
\newblock American Mathematical Society, Providence, RI, 1999.

\bibitem[Hov01]{hovey-spectra}
Mark Hovey.
\newblock Spectra and symmetric spectra in general model categories.
\newblock {\em J. Pure Appl. Algebra}, 165(1):63--127, 2001.

\bibitem[HW13]{hovey-white}
Mark Hovey and David White.
\newblock An alternative approach to equivariant stable homotopy theory,
  preprint available electronically from \texttt{arxiv:1312.3846}.
\newblock 2013.

\bibitem[Ill83]{illman} S. Illman, The Equivariant Triangulation Theorem for Actions of Compact Lie Groups, Math. Ann. 262, 487-501 (1983).

\bibitem[LMSM86]{lewis-may-steinberger}
L.~G. Lewis, Jr., J.~P. May, M.~Steinberger, and J.~E. McClure.
\newblock {\em Equivariant stable homotopy theory}, volume 1213 of {\em Lecture
  Notes in Mathematics}.
\newblock Springer-Verlag, Berlin, 1986.
\newblock With contributions by J. E. McClure.

\bibitem[L{\"u}c05]{luck-survey-classifying-families}
Wolfgang L{\"u}ck.
\newblock Survey on classifying spaces for families of subgroups.
\newblock In {\em Infinite groups: geometric, combinatorial and dynamical
  aspects}, volume 248 of {\em Progr. Math.}, pages 269--322. Birkh\"auser,
  Basel, 2005.

\bibitem[LU14]{luck-uribe}
Wolfgang L\"{u}ck and Bernardo Uribe.
\newblock Equivariant principal bundles and their classifying spaces.
\newblock {\em Algebraic and Geometric Topology}, 14:1925--1995, 2014.

\bibitem[MM02]{mandell-may-equivariant}
M.~A. Mandell and J.~P. May.
\newblock Equivariant orthogonal spectra and {$S$}-modules.
\newblock {\em Mem. Amer. Math. Soc.}, 159(755):x+108, 2002.

\bibitem[MMSS01]{MMSS}
M.~A. Mandell, J.~P. May, S.~Schwede, and B.~Shipley.
\newblock Model categories of diagram spectra.
\newblock {\em Proc. London Math. Soc. (3)}, 82(2):441--512, 2001.

\bibitem[May97]{May_def_operads}
J.~P. May.
\newblock Definitions: operads, algebras and modules.
\newblock In {\em Operads: {P}roceedings of {R}enaissance {C}onferences
  ({H}artford, {CT}/{L}uminy, 1995)}, volume 202 of {\em Contemp. Math.}, pages
  1--7. Amer. Math. Soc., Providence, RI, 1997.

\bibitem[McC96]{mcclure-e-infinity-tate} J.E. McClure. {$E_\infty$}-ring structures for {T}ate spectra. {\em Proceedings of the AMS} 124, 1917--1922, 1996.

\bibitem[Pia91]{piacenza} R.J. Piacenza. Homotopy theory of diagrams and {CW}-complexes over a category. Can. J. Math.Vol. 43 (4), 814-824, 1991.

\bibitem[Qui67]{Quillen}
D. Quillen.
\newblock \textit{Homotopical Algebra}.
\newblock Lecture Notes in Math.
\newblock 43, Springer-Verlag Berlin, 1967.

\bibitem[Rez96]{rezk-phd}
Charles Rezk.
\newblock Spaces of algebra structures and cohomology of operads.
\newblock 1996.
\newblock Thesis (Ph.D.)--MIT.

\bibitem[Rub17]{rubin}
Jonathan Rubin.
\newblock On the realization problem for ${N}_\infty$-operads, preprint available as \texttt{arxiv:1705.03585}, 2017.

\bibitem[SW03]{salvatore-wahl}
Paolo Salvatore and Nathalie Wahl.
\newblock Framed discs operads and {B}atalin-{V}ilkovisky algebras.
\newblock {\em Q. J. Math.}, 54(2):213--231, 2003.

\bibitem[SS00]{SS00}
Stefan Schwede and Brooke~E. Shipley.
\newblock Algebras and modules in monoidal model categories.
\newblock {\em Proc. London Math. Soc. (3)}, 80(2):491--511, 2000.

\bibitem[Spi01]{spitzweck-thesis}
Markus Spitzweck.
\newblock Operads, algebras and modules in general model categories, preprint
  available electronically from \texttt{arxiv:math/0101102}.
\newblock 2001.

\bibitem[Ste16]{stephan}
Marc Stephan.
\newblock On equivariant homotopy theory for model categories.
\newblock {\em Homology, Homotopy, and Applications}, (18):2, 183--208, 2016.

\bibitem[Whi14a]{white-localization}
David White.
\newblock Monoidal bousfield localizations and algebras over operads, preprint
  available electronically from \texttt{arxiv:1404.5197}.
\newblock 2014.

\bibitem[Whi14b]{white-commutative-monoids}
David White.
\newblock Model structures on commutative monoids in general model categories. arXiv:1403.6759.
\newblock {\em Journal of Pure and Applied Algebra}, Volume 221, Issue 12, 2017, Pages 3124-3168.


\bibitem[WY15a]{white-yau1}
David White and Donald Yau.
\newblock {B}ousfield localizations and algebras over colored operads, 
\newblock {\em Applied Categorical Structures}, 26:153--203, 2018.
Available as arxiv:1503.06720.

\bibitem[WY16]{white-yau3}
David White and Donald Yau, Homotopical adjoint lifting theorem, arXiv:1606.01803.
\end{thebibliography}
\end{document}